%% file: Torus_Shock_v6.tex
\documentclass[12pt]{amsart} \usepackage{latexsym, amsmath,
  amssymb,amsthm,amsfonts,amsopn} \usepackage{subfigure}
\usepackage{epsfig,graphics,color,graphicx,graphpap} \usepackage{url}
\usepackage{hyperref}

\include{gsmacros}

\title[Discrete Rarefaction Waves]{On discrete rarefaction waves in an
  NLS Toy Model for weak turbulence}

\author[S.~Herr]{Sebastian~Herr}

\address{Universit\"at Bielefeld, Fakult\"at f\"ur Mathematik,
  Postfach 10 01 31, 33501 Bielefeld, Germany}
\email{herr@math.uni-bielefeld.de}

\author[J.~L.~Marzuola]{Jeremy~L.~Marzuola} \address{Mathematics
  Department, University of North Carolina - Chapel Hill, Chapel Hill,
  NC 27599, USA} \email{marzuola@math.unc.edu}

\begin{document}

\begin{abstract}
  Recently, a model Hamiltonian dynamical system has been derived to
  study frequency cascades in the cubic defocusing nonlinear
  Schr\"odinger equation on the torus.  Here, we explore certain
  rarefaction wave-like solutions to this toy model which transfer
  mass from low to high modes.
\end{abstract}

\maketitle

% \tableofcontents

\section{Introduction}
\label{s:intro}
In \cite{CKSTT} the authors Colliander-Keel-Staffilani-Takaoka-Tao
studied the $2d$ defocusing cubic toroidal nonlinear Schr\"odinger
equation,
\begin{equation}
  \label{e:dcnls}
  i u_t + \Delta u - |u|^2 u = 0, \ \ u(0,x) = u_0(x) \ \text{for} \ x \in \mathbb{T}^2,
\end{equation}
by developing their ``Toy Model System'' given by the equation
\begin{equation}
  \label{e:toy_model}
  -i\dt b_j (t) = -\abs{b_j(t)}^2 b_j(t) + 2 b_{j-1}^2 \overline{b_j}(t)
  + 2 b_{j+1}^2 \overline{b_j}(t),\quad \text{for $j=1\ldots N$}
\end{equation}
with ``boundary conditions''
\begin{equation}
  \label{e:dirichletbc}
  b_0(t) = b_{N+1}(t) = 0.
\end{equation}
The $b_j$ approximate the mass associated with families of resonantly
interacting frequencies.  In \cite[Section 5]{CMOS1}, the authors
derive a discrete Burgers type equation with a phase drag term (see
\eqref{e:hydro1}, \eqref{e:hydro2} below) and study its numerical
stability within \eqref{e:toy_model}.

The goal of this paper is to show that a discrete rarefaction wave
associated to the Burgers equation can be used to initiate interesting
dynamics in \eqref{e:toy_model}. In particular, the aim is to prove
that this mechanism transfers mass from low to high frequency
nodes on a short time scale.
% this is distinctly possible due to an
%explicit construction of a solution to a discrete Burgers equation
%recently posted in works of Ben-Naim et al \cite{B-NK1,B-NR1}.
In addition, while much of the global structure of the rarefaction
wave-like solutions in \eqref{e:toy_model} remains challenging to
describe fully, we present several computations that give insight into
the longer time behavior of discrete rarefaction wave solutions as
observed in \cite{CMOS1} and are consistent with further mass
transfer.

The main goal of developing \eqref{e:toy_model} in \cite{CKSTT} is the
construction of a solution to \eqref{e:toy_model} which transfers mass
from low index $j$ to high $j$. In other words, the goal is to
robustly construct frequency cascades to show that, as stated in
\cite{CMOS1}, ``dispersive equations posed on tori have weakly
turbulent dynamics; while there may be no finite time singularity,
arbitrarily high index Sobolev norms can grow to be arbitrarily large,
but finite, in finite time.''  We note that here we are focusing on
the dynamical system in \eqref{e:toy_model} and attempting to
ascertain how robust the rarefaction wave structure is under the
``phase drag'' inherent to dispersive Schr\"odinger models and the Toy
Model in particular.  For other works related to frequency cascades
and the study of weak turbulence for NLS, we refer the reader to
\cite{bourgain1995cauchy,bourgain1995aspects,bourgain1996growth,B04}
as well as the interesting and recent works
\cite{Carles:2012jv,Hani2,GK12,Ku1,Soh1,GT1,grebert2012beating,guardia2012growth,haus2012dynamics,faou2013weakly,hani2013modified}.

The paper will proceed as follows.  In Section \ref{s:conlaw}, we
recall some conserved quantities related to the Toy Model to be
applied later.  Then, in Section \ref{s:discrar}, we recall the
modified discrete Burgers equation and corresponding rarefaction wave
approximation using the Madelung transformation from \cite{CMOS1}.  We
proceed in Section \ref{s:rarefaction} to discuss properties of
rarefaction wave solutions to a discrete Burgers equation, drawing
largely from an explicit solution in \cite{B-NK1,B-NR1}, and study
boundary effects in symmetric discrete Burgers equation. In Section
\ref{sect:pert} we prove an error bound for the discrete Burgers
rarefaction wave in \eqref{e:toy_model} rigorously (and arguably
sharply) using a Gronwall type argument.  Finally, in Section
\ref{sect:rem}, we present flux computations related to truncated
conservation laws, discuss future work, open problems and an
illustrative computation about the rarefaction wave linearization in
discrete $L^2$ spaces, which we hope will provide for more robust
control of mass transfer through rarefaction waves in the Toy Model.

\section{Conserved Quantities}
\label{s:conlaw}
As they will be quite useful in our studies below, we recall here the
results from \cite[Section 3]{CKSTT}, where the Toy Model is studied
as a Hamiltonian dynamical system. The Hamiltonian is given by
\begin{equation}
  \label{e:hamiltonian}
  H [{\bf b}] = \sum_{j=1}^\infty \left( \frac12 | b_j |^4 - 2 \Re ( \bar{b}_j^2 b_{j-1}^2) \right)
\end{equation}
and symplectic structure,
\begin{equation}
  \label{e:symplectic}
  i \frac{db_j}{dt} = \frac{\partial H [{\bf b}]}{\partial \bar b_j},
  \quad j\in \mathbb{N}.
\end{equation}
The model \eqref{e:toy_model} admits many of the symmetries of
\eqref{e:dcnls}, including phase invariance, scaling, time translation
and time reversal.  However, many of these symmetries are redundant,
and the known only invariant, other than \eqref{e:hamiltonian}, is the
mass quantity
\begin{equation}
  \label{e:truncmass}
  M [{\bf b} ]=  \sum_{j=1}^\infty |b_j|^2.
\end{equation}

Using the structure of the equations, it can be seen that given $b(0)$
initially compactly supported (on a finite number of nodes), the
solution $b(t)$ remains compactly supported for all time.  While we
are summing over all nodes above to observe conservation of $H$ and
$M$, as the equation with boundary conditions \eqref{e:dirichletbc}
remains compactly supported on the same region, both $H$ and $M$ are
still conserved when one sums only from $j = 1$ to $N$.

Defining circles $\mathbb{T}_j$ for $j = 1,\dots,N$ as
\begin{equation*}
  \mathbb{T}_j = \{ {\bf b} = (b_1, \dots, b_N) \ |  \ |{\bf b}|^2 = 1, \  |b_j| =1, \ b_k = 0 \ \text{for all} \ k \neq j \},
\end{equation*}
the authors in \cite{CKSTT} point out that the flow of
\eqref{e:toy_model}, which is referred to as $S(t) b_0$, leaves each
${\mathbb T}_j$ invariant.  In \cite{CMOS1}, it was observed that
\eqref{e:toy_model} also has a natural probabilistic formulation and
can be seen to have some basic recurrence properties.

\section{Discrete Rarefaction Waves and the Discrete Burgers Equation}
\label{s:discrar}
In this section, we recall the main aspects of \cite[Section
5]{CMOS1}. First, we make the Madelung transformation given by
\begin{equation}
  \label{e:mad}
  b_j (t) = \sqrt{ \rho_j} e^{ i \phi_j (t) }.
\end{equation}
with \emph{out of phase} initial interactions set by $\phi_j (0)=
\phi_{j-1} (0) + \frac{\pi}{4}$.  Initially, the hydrodynamic
equations have a Burgers type structure
\begin{equation}
  \label{e:burgers}
  \left\{ \begin{array}{l}
      \dot\phi_j = 0  \\
      \dot\rho_j = -4 \rho_j \rho_{j-1}  - 4 \rho_j \rho_{j+1}  = -8 \rho_j \left( \frac{ \rho_{j+1} - \rho_{j-1} }{ 2} \right).
    \end{array} \right. 
\end{equation}
This system has beautiful discrete rarefaction waves propagating
towards infinity and a backwards dispersive shock. We also refer the
reader to \cite{HerrmannRademacher} for another example of a discrete
conservation law with a dispersive shock, cf.\ Subsection
\ref{subsect:other}.  We call this the discrete Burgers equation since
in the continuum limit we would have
\begin{equation*}
  \rho_t = -8 \rho \nabla \rho,
\end{equation*}
which, with initial data
\begin{equation*}
  \rho (0,x) = \left\{  \begin{array}{ll}
      0 \ & \  x \leq 0, \\
      1 \ & \  0 < x <  \infty,
    \end{array}  \right.
\end{equation*}
has the known rarefaction wave solution
\begin{equation*}
  \rho (t,x) = \left\{  \begin{array}{ll}
      0 \ & \  x < 0, \\
      \frac{x}{8t} \ & \  0 < x <  8t, \\
      1 \  & \ 8t < x.
    \end{array}  \right.  
\end{equation*}

However, in our discrete system, there is drag in the phase
coefficients that does not allow us to permanently assume an out of
phase framework:
\begin{align}
  \label{e:hydro1}
  \dot\phi_j & =  -\rho_j + 2 \rho_{j-1} \cos(2(\phi_{j-1}-\phi_j)) + 2 \rho_{j+1} \cos(2(\phi_{j+1}-\phi_j)),  \\
  \dot\rho_j & = -4 \rho_j \rho_{j-1} \sin(2(\phi_{j-1}-\phi_j)) -4
  \rho_j \rho_{j+1} \sin(2(\phi_{j+1}-\phi_j)). \label{e:hydro2}
\end{align}
Let us recall some numerical simulations from \cite{CMOS1} to motivate
our analysis below.  There, it is numerically studied how solutions
evolve, with the initial condition
\begin{equation}
  \label{e:shock_ic}
  b_j = \exp i \set{ (j-1)\pi/4}.
\end{equation}
In Figure \ref{f:shockshort_compare} we show the numerically computed
time evolution of the Toy Model compared to that of a backward
discrete Burgers equation.

 \begin{figure}
   \subfigure{\includegraphics[width=5cm,type=pdf,ext=.pdf,read=.pdf]{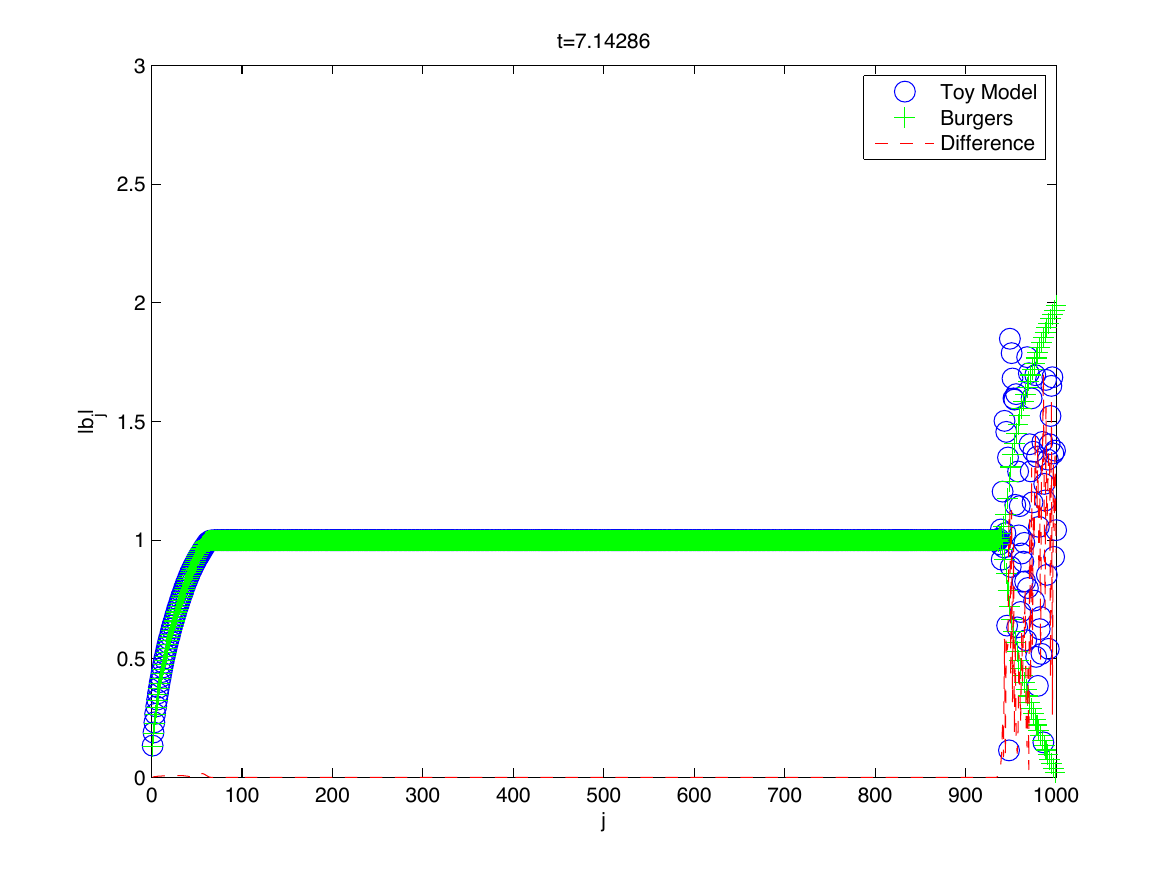}}
   \subfigure{\includegraphics[width=5cm,type=pdf,ext=.pdf,read=.pdf]{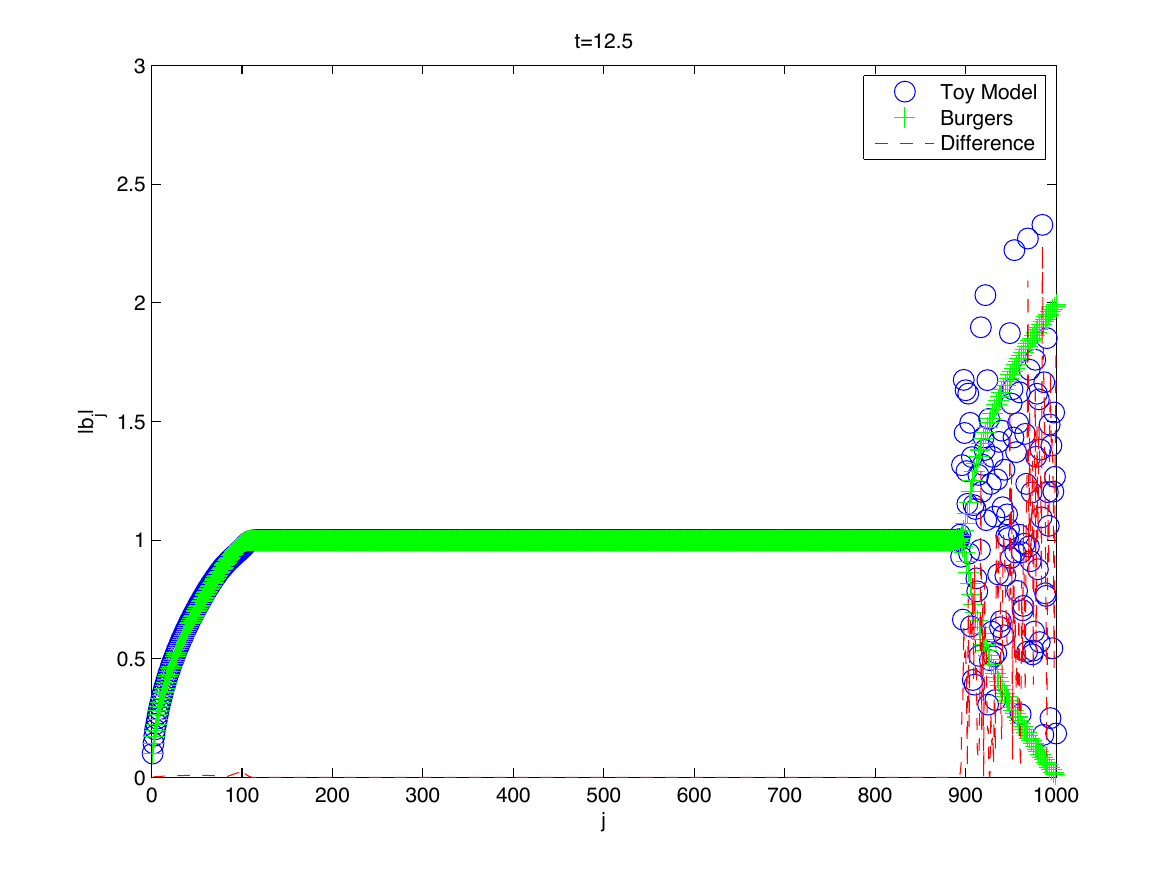}}
   \subfigure{\includegraphics[width=5cm,type=pdf,ext=.pdf,read=.pdf]{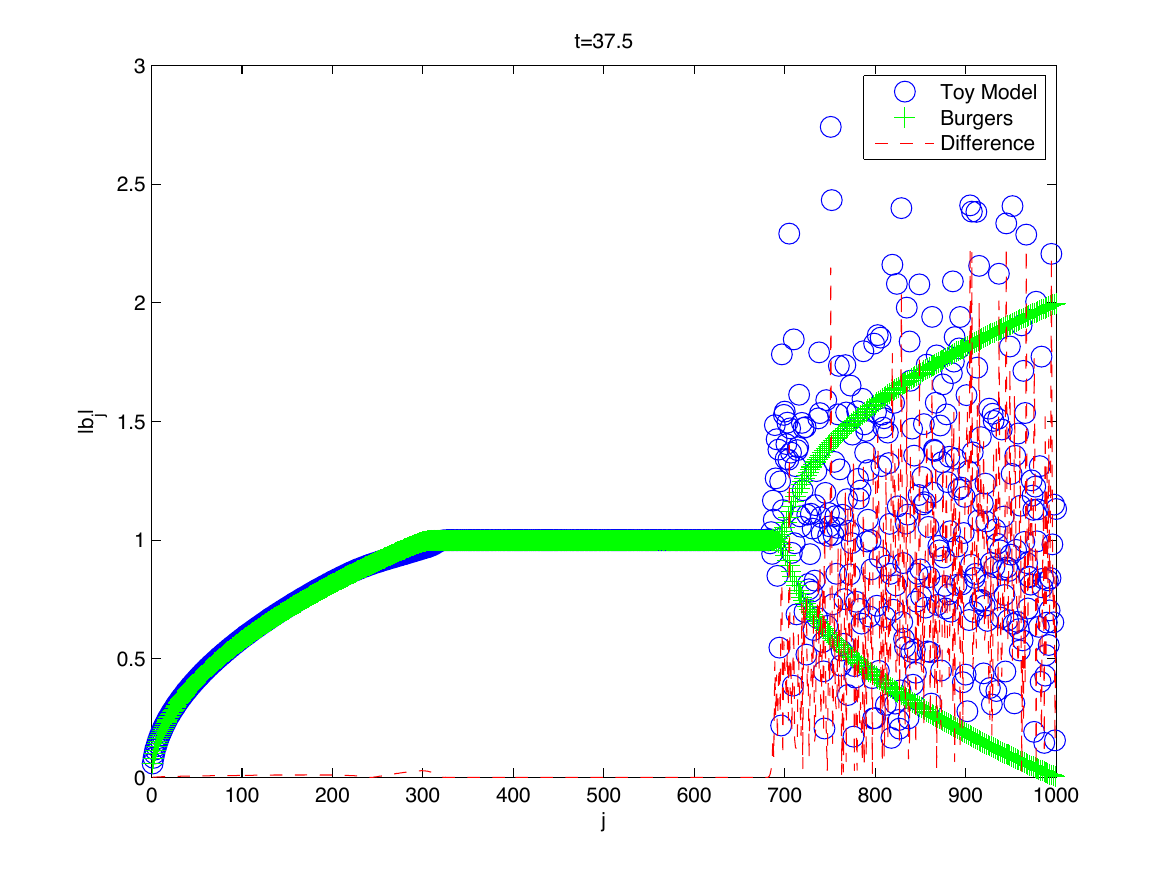}}
   \subfigure{\includegraphics[width=5cm,type=pdf,ext=.pdf,read=.pdf]{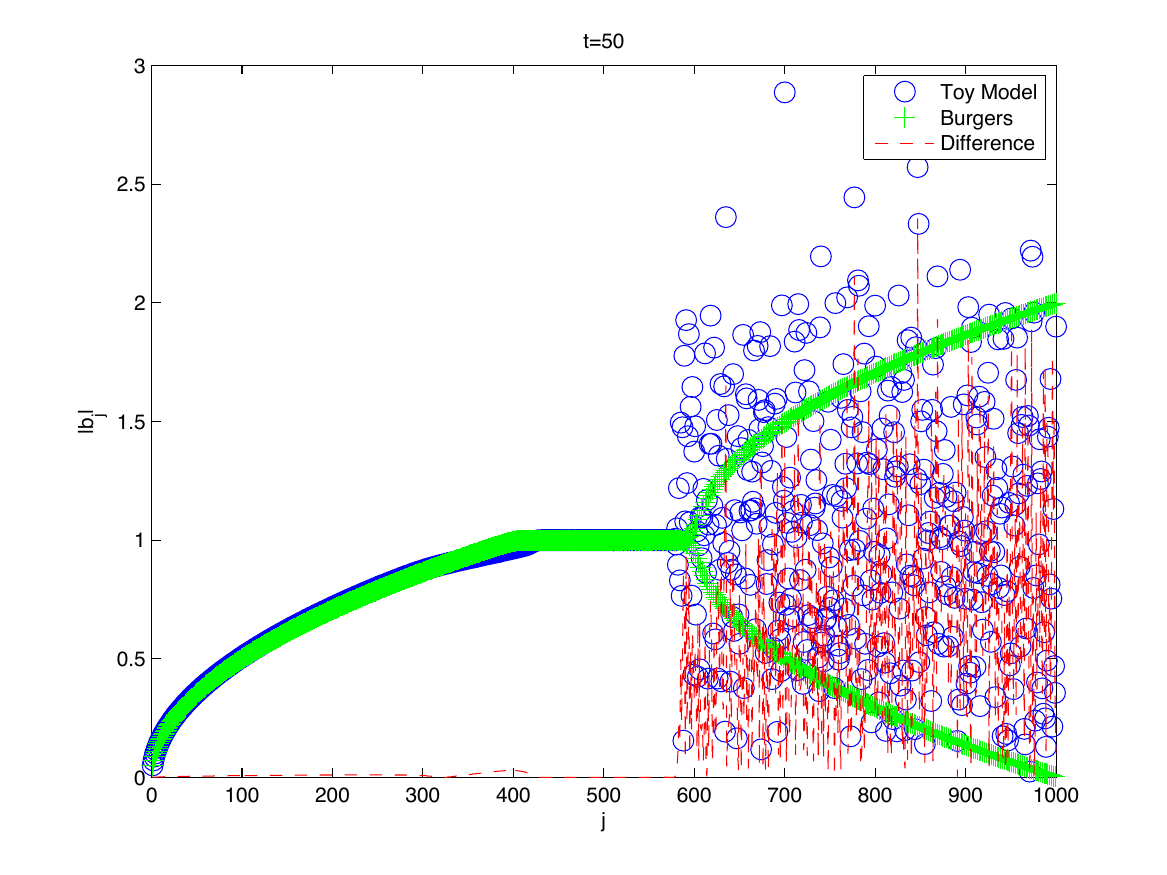}}
   \caption{A comparison of the Toy Model and the symmetric discrete Burgers
     rarefaction waves.}
   \label{f:shockshort_compare}
 \end{figure}

 \section{Analysis of the Rarefaction Wave}
 \label{s:rarefaction}

 In this section, we present some analytic results on the rarefaction
 wave.  In particular, we compare solutions to the discrete Toy Model
 in the hydrodynamic formulation to computations from an explicit
 solution to a discrete Burgers Equation, which behaves comparably to
 a continuous Burgers equation.

 \subsection{Alternative Coordinates}\label{subsect:alt}

 Since the drag term in the phase (the term $-\rho_j$ in
 \eqref{e:hydro1} and $-i |b_j|^2 b_j$ in \eqref{e:toy_model})
 introduces errors in the Burgers evolution, we introduce the
 coordinate $\theta_j = \phi_j - \phi_{j-1}$.  In the new coordinates,
 we have
 \begin{align}
   \label{e:hydroalt1}
   \dot\theta_j & =  -(\rho_j-\rho_{j-1}) - 2 ( \rho_j - \rho_{j-1}) \cos(2 \theta_j) \\
   &\hspace{1cm}+ 2 \rho_{j+1} \cos(2 \theta_{j+1})  - 2 \rho_{j-2} \cos(2 \theta_{j-1}) ,  \notag \\
   \dot\rho_j & = 4 \rho_j \rho_{j-1} \sin(2\theta_j) -4 \rho_j
   \rho_{j+1} \sin(2 \theta_{j+1}). \label{e:hydroalt2}
 \end{align}

 \subsection{Scaling Discrete Burgers}
 \label{subsect:scaling}

 We wish to use an exact solution to a discrete Burgers equation in order to establish an envelope solution to the Toy Model.  The envelope solution we wish to follow is for the symmetric Burgers equation
  \begin{equation}
   \label{e:symburgers}
   \dot{\tilde{\rho}}_j =  - \tilde{\rho}_j \left(  \tilde{\rho}_{j+1} - \tilde{\rho}_{j-1} \right).
 \end{equation} 
 It is possible to solve such an equation using inverse scattering (see \cite{GoodmanLax,GHSZ-Toda,KvM} as we will discuss later in Section \ref{s:numerics}), but a priori bounds on the amplitude then become less clear than is desirable for comparison with the Toy Model.  
The
 best treatment of which we have found in the works
 \cite{B-NK1,B-NR1}, where for the backward discrete Burgers
 \begin{equation}
   \label{e:dbackburgers}
   \dot{p}_j =  - p_j \left(  p_{j} - p_{j-1} \right),
 \end{equation} 
 an explicit solution is derived.  To do so, they introduce the
 transformation $p_j = \frac{a_j}{a_{j+1}}$, and the
 problem is converted to the recursively solved linear system of ODEs
 \begin{equation*} {\dot a}_j = a_{j-1}.
 \end{equation*}
 With initial data configured only for a rarefaction wave solution
 \begin{equation}
   \label{e:initraref}
   p_j (0) = \left\{  \begin{array}{l}
       0 \  \  j \leq 0, \\
       1 \ \  1\leq j <  \infty,
     \end{array}  \right.  
 \end{equation} 
 this has a solution of the form
 \begin{equation*}
   a_j(t) =  (1 +  t  + \dots + \frac{1}{(j-1)!} t^{j-1}).
 \end{equation*}
 Note, this solution does not have finite speed of propagation, in
 that every non-zero $q_j$ will change for all $t>0$, not just those
 to the left of a front as in the continuous shock solution.  However,
 the resulting solution in \eqref{e:dbackburgers} is
 \begin{align*}
   p_j (t) & =  \frac{ 1 + t  + \dots + \frac{1}{(j-1)!} t^{j-1} }{ 1 + t  + \dots + \frac{1}{(j)!} t^{j} }   =  1 - \frac{ \frac{1}{(j)!} t^{j} }{ 1 + t  + \dots + \frac{1}{(j)!} t^{j} } \\
   & = 1 - \frac{1}{(j)!} t^{j} + h.o.t.
 \end{align*}
 for short times $t$.  Then, to solve the rescaled Burgers equation
 \begin{align}
   \label{e:scaleddbackburgers}
   \dot{p }_j (\alpha,\beta, t) = -\alpha p_j
   \left( p_{j} - p_{j-1} \right), \
   p_j (\alpha, \beta, 0) = \left\{ \begin{array}{l}
       0 \  \  j \leq 0, \\
       \beta \ \ 0 < j < \infty,
     \end{array}  \right.  
 \end{align}
 we have
 \begin{equation*} {p}_j (t)= \beta \frac{ 1 + \alpha \beta t +
     \dots + \frac{1}{(j-1)!} (\alpha \beta t)^{j-1} }{ 1 + \alpha
     \beta t + \dots + \frac{1}{(j)!} (\alpha \beta t)^{j} } .
 \end{equation*}

 Generally, in order to control both the dispersive shock\footnote{In
   discussions with Mark Hoefer, he has suggested that this backward
   propagating wave takes on more features of turbulence than of a
   dispersive shock front. However, to investigate this involves
   connecting the oscillations from the front of the backward
   propagating front in a more robust fashion.} created by a large
 jump near $j=N$ and the phase splitting mechanism at large amplitude
 for lattice sites near $j=1$, the parameter $\beta$ will be set by us
 to simply be $\epsilon/8$ to study rarefaction waves in the Toy
 Model, where the $8$ is a scaling parameter chosen for convenience.
 Indeed, due to the scalings above and the nature of the Toy Model, we
 study the following ``backwards'' discrete Burgers equation,
 \begin{equation}
   \label{e:dburgers}
   \dot{p}_j =  -8 p_j \left(p_{j} - p_{j-1} \right).
 \end{equation}
 We have via the remarkable explicit
 solution from \cite{B-NK1,B-NR1}, the rescaled solution
 \begin{equation}
   \label{e:scaledburgers} \begin{split}{p}_j &= \frac{\epsilon}{8} \frac{
     1 + (\epsilon t) + \dots + \frac{1}{(j-1)!} (\epsilon t) ^{j-1}
   }{ 1 + (\epsilon t) + \dots + \frac{1}{(j)!} (\epsilon t) ^{j} }\\
& =
   \frac{1}{8} \frac{ \partial_t (1 + (\epsilon t) + \dots +
     \frac{1}{(j)!} (\epsilon t) ^{j}) }{ 1 + (\epsilon t) + \dots +
     \frac{1}{(j)!} (\epsilon t) ^{j} } .
 \end{split}
\end{equation}

 While we make use of this rescaling in order to get natural smallness
 in the phase drift term, we can always rescale the solution back to
 order $1$ by the same argument, see Remark \ref{rem:rescale} for a
 further discussion.

 Now, as a leading order description of the behavior in the full Toy
 Model, we propose the following modified discrete Burgers equation
 for capturing the dynamics.
 \begin{equation}
   \label{e:modburgers}
   \left\{ \begin{array}{l}
       \dot {\tilde{\theta}}_j = -(\tilde{\rho}_j-\tilde{\rho}_{j-1})  \\
       \dot {\tilde{\rho}}_j =  \tilde{\rho}_j \tilde{\rho}_{j+1}  - 4 \tilde{\rho}_j \tilde{\rho}_{j}  = -4 \tilde{\rho}_j \left( \tilde{\rho}_{j+1} - \tilde{\rho}_{j-1} \right).
     \end{array} \right. 
 \end{equation}
 Note that these equations are completely decoupled.

 We will show in our analysis of \eqref{e:toy_model} in Section
 \ref{sect:pert} that errors arising around this approximation are
 small on the time scales we study.  It is however the errors in phase
 term that account primarily for the slight deviations from the pure
 rarefaction wave on the left and the approximation of discrete
 symmetric Burgers by discrete backwards Burgers on the right in
 Figure \ref{f:shockshort_compare}.

\subsection{A Priori Estimates for Rarefaction Waves}\label{subsect:gronw1}
In this section, we want to prove a priori bounds on the evolution of \eqref{e:modburgers} based upon a leading order approximation using \eqref{e:dbackburgers}.  We focus on the amplitude equations.  Specifically, we will approximate solutions to
\[
\dot \trho_j = -4 \trho_j (\trho_{j+1} - \trho_{j-1}) \ \ \trho_j (0) = \epsilon/8 \ \text{for} \ j = 1, \dots , N,
\]
by the explicit rarefaction wave solution defined in \eqref{e:scaledburgers} to
\begin{equation}
\label{peq}
\dot p_j = -8 p_j (p_{j} - p_{j-1}), \ \ p_j (0) = \epsilon/8 \ \text{for} \ j = 1, \dots , N.
\end{equation}
Setting $q_j = \trho_j -p_j$ and $\sigma_j = p_{j+1}-p_j$, we observe
\begin{equation}
\label{qeq}
\begin{split}
\dot q_j ={}& -4 p_j (\sigma_{j} - \sigma_{j-1})  - 4 q_j (\sigma_j
+\sigma_{j-1})\\
&- 4 p_j (q_{j+1} - q_{j-1}) - 4  q_j (q_{j+1} - q_{j-1}), \\ 
q_j (0) ={}& 0.
\end{split}
\end{equation}
Let us define the forcing component of this system of ODEs as
\[
F_{j} = -4 p_j (\sigma_{j} - \sigma_{j-1}) .
\]
Moving the term $4 q_j (\sigma_j+\sigma_{j+1})$ to the left hand side and applying an integrating factor argument, we obtain
\begin{align}
  \label{eqn:approx-ineq1}
| q_j | (t)   \leq &  \int_0^t e^{-4 \int_s^t (\sigma_{j} +\sigma_{j-1})(s') ds'  }  \left( | F_j |   + 8 \| q \|_{\infty}   |p_{j} |   +  8 \| q \|_{\infty}   | q_j |   \right)  ds 
\end{align}
uniformly on the time interval $[0,\epsilon^{-1} \delta]$. Here and in
the sequel we use the notation $\|q\|_{\infty}:=\sup_{0\leq t \leq T}\sup_{ j \in \Z}|q_j(t)|$.

%6 p q \leq 6 ( (6T)/2 p^2 + 1/(2 (6T)) q^2

The following lemma contains the a priori estimate for $\|q\|_{\infty}$.
\begin{lem}
  Fix $\delta < 1/8$. For all $0<\epsilon \leq \delta$ and $T= \epsilon^{-1} \delta$ and $q_j$, $p_j$ as in \eqref{peq}, \eqref{qeq} respectively, we have
  \begin{equation}
  \label{apqbd}
     \| q \|_{\infty} \leq \frac{\epsilon\delta}{2} .
     \end{equation}
\end{lem}

\begin{proof}
The proof follows by a simple bootstrap argument on
\eqref{eqn:approx-ineq1}.  Let us note that the exponential factor is
uniformly bounded by $1$ for every $j$ except at $j=N$
where it is bounded by $e^{\delta}$ on our time scale.  Then, from the
bootstrapping assumption we obtain
\begin{equation*}
\| q \|_{\infty} \leq e^{\delta} \left( 16 \frac{\epsilon^2}{8^2} T + \epsilon^2 \delta T/2 + 2 \epsilon^2 \delta^2 T \right),
\end{equation*}
which is bounded above by 
\[ 
\frac{\epsilon \delta}{2} \left(  e^{\delta} \left[  \frac12 + \frac{\delta}{2} + 4 \epsilon \delta^2 \right] \right) < \frac{\epsilon \delta}{2}
\]
for any $\delta < 1/8$, provided $\epsilon \leq \delta$.  
\end{proof}

Armed with this a priori estimate, we can get refined estimates for $|q_j (t)|$ for $0 \leq j \leq N-1$ using the fact that the exponential factor is bounded by $1$ in these cases.  Indeed, a direct calculation using the explicit solution on time scale $T$ shows that for instance $|F_1| < \epsilon^2/8^2 (1- \epsilon t)^2$.  In turn, we can show
\[
|q_1 (t) | \leq \epsilon \delta ( \frac{1}{2^4} + \delta + 8 \delta^2) < \frac{ \epsilon \delta}{12}
\]
for $\delta$ chosen sufficiently small.  Even stronger estimates hold for $0 < j < N$.

A symmetric argument using the \eqref{qeq} for $q_N$ shows that
\[
|q_N (T) | \geq \frac{ \epsilon \delta}{16},
\]
by recognizing that $F_N \approx 4 \rho_N^2$ and $\sigma_N = 4 p_{N-1}$.  We will show in Section \ref{subsect:shock} from a different argument that $\rho_N$ is increasing in fact, but without such quantitative bounds as we are able to gain from \eqref{qeq}.

\begin{rem}
  \label{rem:rescale}

  By nature of the construction, we now have that
  \begin{equation*}
     \rho_0 (t) \leq p_j (t) + \left(\frac{\epsilon \delta}{12} \right).
  \end{equation*}
  Hence, near the end of our evolution the rarefaction wave at the
  left of the grid has size roughly
  \begin{equation*}
    p_0 (T) = \frac{\epsilon}{8} \frac{1}{1 + \epsilon T} \approx \frac{ \epsilon }{8}(1 - \delta),\end{equation*}
  which is less than the initial amplitude $\epsilon/8$ even when compared to the error.  Hence, we observe that what we have seen in Figures \ref{f:shockshort_compare}, namely that the rarefaction wave solution moves mass to the right initially in symmetric Burgers.  Similar estimates can be proven using the stronger bounds on $q_j$ for $j > 1$, though on the time scale $T$ most nodes have moved very little.  While we are not arguing this makes the rarefaction wave solution a good approximation for long times, we can observe both that the rarefaction wave initializes a motion of mass towards the right for the low modes.  Numerically, we observe that the rarefaction wave behavior is much more robust than we can fully understand at the moment.    
  \end{rem}

 \subsection{The Shock vs.\ Rarefaction in the Toy
   Model}\label{subsect:shock}
 Let us analyze the symmetric discrete Burgers equation in
 \eqref{e:burgers} with initial condition
 \begin{equation}
   \label{e:initdatasymburgers}
   \trho_j = 1 \ \text{for} \ 1 \leq j \leq 2 N, \ \ \trho_j = 0 \ \text{otherwise}.
 \end{equation}  
 Then, we observe that such an equation can be decomposed into a
 coupled system of equations for $s_j = \trho_{2j}$ and $r_{j} =
 \trho_{2j+1}$, which results in the system
 \begin{equation*}
   \dot s_j = -4 s_j (r_j - r_{j-1}) , \ \ 
   \dot r_j = -4 r_j (s_{j+1} - s_{j}).
 \end{equation*}
 Now, if we look at the right-most points, we observe
 \begin{equation*}
   \dot s_N = 4 s_N (r_{N-1}) , \ \ \dot r_{N-1} = -4 r_{N-1} (s_N - s_{N-1}),
 \end{equation*}
 which given that $r_j, s_j > 0$ for $1 \leq j \leq N$ implies that
 $s_N$ is an increasing function.  As a result, this implies that
 $r_{N-1}$ is a decreasing function.  Propagating this down the line
 interactions, we observe that the symmetric Burgers causes a
 splitting from the right endpoint instead of a shock moving right,
 see Figure \ref{f:symdiscburg} for a numerical simulation of this
 effect.  Note that the leading order component on the left is still
 that of a rarefaction wave however.  It is indeed this wave front we
 believe acts as the envelope for the Toy Model rarefaction wave,
 however as we do not have explicit control on its evolution, we focus
 on the rarefaction wave coming from the appropriate backwards Burgers
 evolution as in \eqref{e:dbackburgers}.

 \begin{figure}
   \includegraphics[width=7cm,type=pdf,ext=.pdf,read=.pdf]{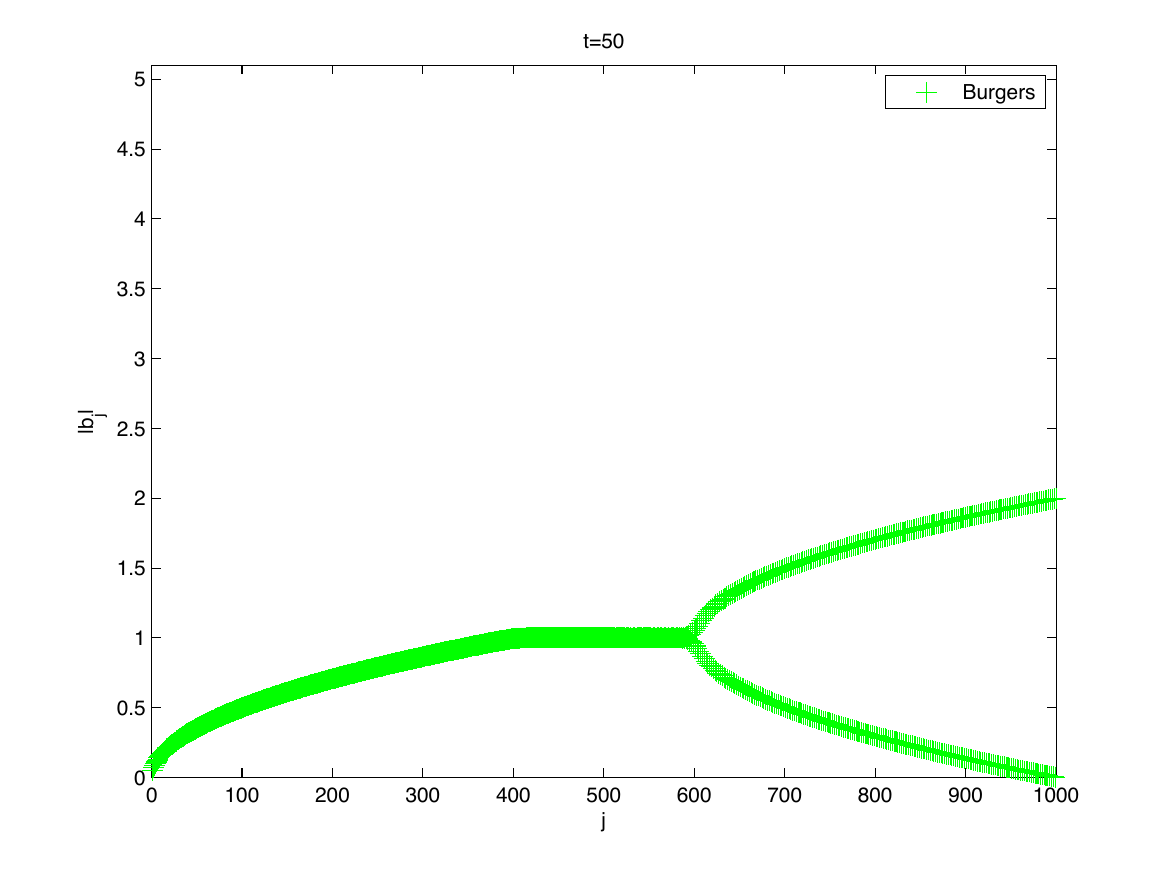}
   \caption{The left moving split that generates the dispersive shock
     from the symmetric Burgers evolution of
     \eqref{e:initdatasymburgers}.  This has been evolved to time
     $T=50.0$ on a lattice of size $N=1000$.}
   \label{f:symdiscburg}
 \end{figure}

 \section{Perturbation theory}
 \label{sect:pert}

\subsection{Equations for error terms}\label{subsect:lin}
Let us now fix a lattice with $N$ nodes and explicitly study equations
\eqref{e:hydroalt1} and \eqref{e:hydroalt2} with initial conditions
given by
\begin{equation}
  \label{e:pertid}
  \theta_j (0)  =  \frac{\pi}{4}, \ \ \rho_j (0) = \frac{\epsilon}{8}
\end{equation}
for all $j = 1, \dots, N$, and $0$ otherwise.

We wish to observe what sorts of error terms arise when we perturb
around the dynamics in \eqref{e:modburgers} in the full Toy Model.  To do this, we will first derive a new solution for 
\eqref{e:symburgers} as a perturbation of the explicit solutions explored above for \eqref{e:dbackburgers} in Section \ref{subsect:scaling}.

The difficulty here with directly using \eqref{e:modburgers} is that we must first approximate the existence of a rarefaction wave like solution and prove a priori estimates.  However, we can study the evolution of
\begin{align*}
  \hat{\theta}_j = \theta_j - \tilde{\theta}_j, \ \hat{\rho}_j =
  \rho_j - \tilde{\rho_j}
\end{align*}
by linearizing the $\sin$ and $\cos$ terms in
\eqref{e:hydroalt1}, \eqref{e:hydroalt2}. Before we proceed, recall the expansions
\begin{equation}
  \sin (\pi/2 + x) = 1 - \frac{x^2}{2!} - \dots, \; \cos (\pi/2 + x) = x - \frac{x^3}{3!} + \dots.
\end{equation}

Let us define
\begin{equation}
  \label{eqn:gammabddef}
  - \tgamma_j  =  \int_0^t (\trho_j - \trho_{j-1}) ds.
  \end{equation}
Note that there are much stronger and $j$ dependent bounds on the components stemming from the exact solution $\vec p (t)$,
\begin{align*}
 \int_0^t (p_j - p_{j-1}) ds, &  =  \frac18  \log \left(  \frac{  1 + \epsilon t  + \dots + \frac{1}{(j)!} (\epsilon t)^{j}}{(1 + \epsilon t  + \dots + \frac{1}{(j-1)!} (\epsilon t)^{j-1}) }  \right)   \\
  & =  \frac18  \log \left( 1 + \frac{ \frac{1}{(j)!} (\epsilon t)^{j}}{(1 + \epsilon t  + \dots + \frac{1}{(j-1)!} (\epsilon t)^{j-1}) }  \right)  \\
  & \leq \frac18 \log (1+\frac{ \epsilon t}{j} ) \leq
  \frac{\epsilon t}{ 8 j} \ \ \text{if $\epsilon t \leq 1$}.
\end{align*}
Then using $T = \delta \epsilon^{-1}$ and the corresponding bounds we have from Section \ref{subsect:gronw1}, one has the crude bound
  \begin{equation}
  \label{eqn:gammabd}
| \tgamma_j | \leq  \frac{ \delta}{8} +  2 T \frac{\epsilon \delta}{2}  \leq  \frac{ \delta}{8} + \delta^2.
  \end{equation}
Similarly, we define
\begin{align}
  \label{eqn:sigmabddef}
  \tsigma_j & = (\trho_j - \trho_{j-1}) ,
  \end{align}
  which then using the explicit properties of $\vec p (t)$ and bounds we have from Section \ref{subsect:gronw1}, we have the crude bound
\begin{align}
  \label{eqn:sigmabd}
  |\tsigma_j | &  \leq \frac{ \epsilon }{8 (j-1)!} + \epsilon \delta. 
  \end{align}
%  \\
%   &=\frac{\epsilon}8  \frac{ \frac{1}{(j-1)!} (\epsilon t)^{j-1} }{ 1 + (\epsilon t)  + \dots + \frac{1}{(j-1)!} (\epsilon t)^{j-1} }  \times \notag \\
%  & \hspace{1.0cm} \left( 1 - \frac{(\epsilon t)}{j}  \frac{ 1 + (\epsilon t) + \dots + \frac{1}{(j-1)!} (\epsilon t)^{j-1}}{ 1 + (\epsilon t)  + \dots + \frac{1}{j!} (\epsilon t)^{j} }  \right) , \notag \\
%  &  \leq \frac{ \epsilon }{8 (j-1)!} \ \ \text{if
%    $\epsilon t \leq 1$}, \notag
%\end{align}
These definitions will allow us to simplify the process of collecting (generically small in amplitude) components of
our expansion.

By linearizing $\cos$ in \eqref{e:hydroalt1}, we observe that
\begin{align*}
  \dot{\htheta}_j & =  - (\hrho_j - \hrho_{j-1}) ( 1 + 2  \cos ( \pi/2 + 2 ( \htheta_j + \tgamma_j))) \notag \\
  &  \hspace{0.45cm}- 2 ( \trho_j - \trho_{j-1} ) \cos ( \pi/2 + 2 ( \htheta_j + \tgamma_j)) \notag \\
  &\hspace{0.45cm} + 2( \hrho_{j+1} + \trho_{j+1})  \cos ( \pi/2 + 2 ( \htheta_{j+1} + \tgamma_{j+1}))) \notag \\
  & \hspace{0.45cm}- 2 ( \hrho_{j-2} + \trho_{j-2})  \cos ( \pi/2 + 2 ( \htheta_{j-1} + \tgamma_{j-1})))  \notag \\
  & = - (\hrho_j - \hrho_{j-1})  + 4 [(\trho_j - \trho_{j-1}) + (\hrho_j - \hrho_{j-1})] (\tgamma_j  + \htheta_j) \notag  \\
  &\hspace{0.45cm}- 4 ( \trho_{j+1} + \hrho_{j+1}) (  \tgamma_{j+1} + \htheta_{j+1} )  + 4  ( \trho_{j-2} + \hrho_{j-2}) (  \tgamma_{j-1} + \htheta_{j-1} ) \notag  \\
  & \hspace{0.45cm} + \widetilde{E}_{1,j} ( \htheta + \tgamma,
  \hrho), \notag
\end{align*}
provided that
$
\hsigma_j = \hrho_j - \hrho_{j-1}
$
and
\begin{align*}
   \widetilde{E}_{1,j} ( \htheta + \tgamma,
  \hrho)  &= -2 \hsigma_j [  \cos ( \pi/2 + 2 ( \htheta_j + \tgamma_j))  +  2 ( \htheta_j + \tgamma_j)] \\
  & \hspace{0.45cm} -2 \tsigma_j [  \cos ( \pi/2 + 2 ( \htheta_j + \tgamma_j))  + 2 ( \htheta_j + \tgamma_j)] \\
  & \hspace{0.45cm} + 2 (\hrho_{j+1} + \trho_{j+1}) [  \cos ( \pi/2 + 2 ( \htheta_{j+1} + \tgamma_{j+1}))  +  2 ( \htheta_{j+1} + \tgamma_{j+1})] \\
  &  \hspace{0.45cm}- 2 (\hrho_{j-2} + \trho_{j-2}) [  \cos ( \pi/2 + 2 ( \htheta_{j-1} + \tgamma_{j-1}))  +  2 ( \htheta_{j-1} + \tgamma_{j-1})].
\end{align*}
Defining
\[
  E_{1,j} ( \htheta + \tgamma,
  \hrho)= 4 (\hrho_j - \hrho_{j-1}) \htheta_j - 4 \hrho_{j+1} \htheta_{j+1} + 4 \hrho_{j-2} \htheta_{j-1} 
+\widetilde{E}_{1,j} ( \htheta + \tgamma,
  \hrho),
\]
we obtain
\begin{align}
  \dot{\htheta}_j& =  - (\hrho_j - \hrho_{j-1}) (1 - 4 \tgamma_j) + 4 (\trho_j - \trho_{j-1}) \htheta_j \label{eqn:thetadot}   \\
  & \hspace{0.45cm} - 4\tgamma_{j+1} \hrho_{j+1} -4 \trho_{j+1} \htheta_{j+1}  + 4\tgamma_{j-1} \hrho_{j-2} + 4 \trho_{j-2} \htheta_{j-1} \notag \\
  &  \hspace{0.45cm}+ f_{1,j} (t) + E_{1,j} ( \htheta +
  \tgamma, \hrho) , \notag
\end{align}
where
\begin{align*}
  f_{1,j} (t) & = 4 \tsigma_j \tgamma_j - 4( \trho_{j+1} \tgamma_{j+1} - \trho_{j-2} \tgamma_{j-1}) \\
  & = 4 \tsigma_j \tgamma_j - 4 \trho_{j+1} (\tgamma_{j+1} - \tgamma_{j-1}) - 4\tgamma_{j-1} ( \trho_{j+1} - \trho_{j-2} ).
\end{align*}

Similarly, by linearizing $\sin$ in \eqref{e:hydroalt2} we obtain
\begin{align}
  \dot{\hrho}_j & =  -4 ( \hrho_j + \trho_j) (\hrho_{j+1} + \trho_{j+1})  \sin ( \pi/2 + 2 ( \htheta_{j+1} + \tgamma_{j+1}))) \notag \\
  & \hspace{.45cm} + 4 ( \hrho_j + \trho_j) (\hrho_{j-1} + \trho_{j-1})  \sin ( \pi/2 + 2 ( \htheta_{j} + \tgamma_{j}))) \notag \\
  & \hspace{.45cm} + 8 \tilde{\rho}_j \left( \tilde{\rho}_{j} - \tilde{\rho}_{j-1} \right) \notag \\
  &=   -4 \hrho_j ( \trho_{j+1} - \trho_{j-1}) - 4 \trho_{j} \hrho_{j+1} + 4 \trho_j \hrho_{j-1}  \label{eqn:rhodot}  \\
  & \hspace{.45cm} + E_{2,j} ( \htheta +
  \tgamma, \hrho) ,\notag  % + f_{2,j}
\end{align}
using
%\begin{align*}
%  f_{2,j} (t) & = -4 \trho_j (\tsigma_{j+1} - \tsigma_{j}),
%\end{align*}
%as well as
\begin{align*}
  \widetilde{E}_{2,j}( \htheta + \tgamma,
  \hrho) & =   -4 ( \hrho_j + \trho_j) (\hrho_{j+1} + \trho_{j+1}) [ \sin ( \pi/2 + 2 ( \htheta_{j+1} + \tgamma_{j+1})) -1] \\
  & \hspace{.45cm} + 4 ( \hrho_j + \trho_j) (\hrho_{j-1} + \trho_{j-1})  [\sin ( \pi/2 + 2 ( \htheta_{j} + \tgamma_{j})) -1]  
\end{align*}
and
\[
E_{2,j}( \htheta + \tgamma,
  \hrho)=4\hrho_j\hrho_{j-1}-4\hrho_j\hrho_{j+1}+\widetilde{E}_{2,j}( \htheta + \tgamma,
  \hrho).
\]

At this stage, we set out to explore in what sense a rarefaction
wave-like solution from the backward Burgers equation approximates the
solution to
\begin{align*}
  \dot\theta_j & =  -(\rho_j-\rho_{j-1}) - 2 ( \rho_j - \rho_{j-1}) \cos(2 \theta_j) \\
  &\hspace{0.45cm}+ 2 \rho_{j+1} \cos(2 \theta_{j+1})  - 2 \rho_{j-2} \cos(2 \theta_{j-1}) ,  \notag \\
  \dot\rho_j & = 4 \rho_j \rho_{j-1} \sin(2\theta_j) -4 \rho_j
  \rho_{j+1} \sin(2 \theta_{j+1}). \label{e:hydroalt2}
\end{align*}
with initial data
\begin{equation*}
  \theta_j (0)  =  \frac{\pi}{4}, \ \ \rho_j (0) = \frac{\epsilon}{8}
\end{equation*}
for all $j = 1, \dots, N$, and $0$ otherwise.

\subsection{A Gronwall Estimate}\label{subsect:gronw}
Multiplying \eqref{eqn:thetadot} by $\htheta_j$ and \eqref{eqn:rhodot}
by $\hrho_j$, we obtain
\begin{align}
  \label{eqn:diff-ineq1}
 & \frac12  \frac{d}{dt} \| \htheta \|_{\ell^\infty}^2 \leq 12 \| \trho
  \|_{\ell^\infty} \| \htheta \|_{\ell^\infty}^2 +  8 \| \tgamma \|_{\ell^\infty} \| \hrho
  \|_{\ell^\infty} \| \htheta \|_{\ell^\infty} \\
&\notag + 2 \| \hrho \|_{\ell^\infty} \| 1 - 4 \tgamma \|_{\ell^\infty}   \| \htheta \|_{\ell^\infty}  + \| f_{1}
  \|_{\ell^\infty} \| \htheta \|_{\ell^\infty} 
+ \|E_1( \htheta + \tgamma,
  \hrho) \|_{\ell^\infty} \| \htheta
  \|_{\ell^\infty}, \\\label{eqn:diff-ineq2}
  & \frac12 \frac{d}{dt}  \| \hrho \|_{\ell^\infty}^2  \leq 12\| \trho \|_{\ell^\infty} \| \hrho \|_{\ell^\infty}^2 +   \| E_2 ( \htheta + \tgamma,
  \hrho)\|_{\ell^\infty} \| \hrho
  \|_{\ell^\infty},
\end{align}
uniformly on the time interval $[0, \delta \epsilon^{-1}]$.

The following lemma contains the crucial Gronwall estimates. Applying the relations 
$ab\leq T a^2+ \frac{1}{4T}
b^2$ in the $\htheta$ equations and $ab\leq 2T a^2+ \frac{1}{2T}
b^2$ in the $\hrho$ equations for $T > 0$, we can write down crude bounds directly from
immediate consequence of \eqref{eqn:diff-ineq1} and
\eqref{eqn:diff-ineq2}.  On the time scale $T = \delta \epsilon^{-1}$, we have from above that 
\[ \trho_j \leq \frac{\epsilon}{8} + \frac{\epsilon \delta}{2} \]
and
\[ \tgamma_j \leq \frac{\delta}{8} + \delta^2.  \]

\begin{lem}
  For all $0< t \leq T = \delta \epsilon^{-1}$ we have
  \begin{equation}
    \label{eqn:bootstrap1}
\begin{split}
 &   \| \htheta (t)\|_{\ell^\infty}^2 \leq  C_1 e^{24 \int_0^t \|\trho\|_{\ell^\infty} ds} I_1(t),\qquad \text{ where }\\
&I_1(t)= \left( \int_0^t  (16 \| \tgamma \|_{\ell^\infty} + 2 \| 1 - 4 \tgamma \|_{\ell^\infty}) T \| \hrho  \|_{\ell^\infty}^2 +   2T \|f_1\|_{\ell^\infty}^2 +   2T \|E_1( \htheta + \tgamma,
  \hrho) \|_{\ell^\infty}^2 ds \right) \\
& \hspace{.4cm} \leq  \left( \int_0^t  ( 2 + 3 \delta + 24 \delta^2) T \| \hrho  \|_{\ell^\infty}^2 +   2T \|f_1\|_{\ell^\infty}^2 +   2T \|E_1( \htheta + \tgamma,
  \hrho) \|_{\ell^\infty}^2 ds \right) ,
\end{split}
\end{equation}
and
  \begin{equation}
    \label{eqn:bootstrap2}\begin{split}
    & \| \hrho (t)\|_{\ell^\infty}^2 \leq C_2 e^{24 \int_0^t \|\trho\|_{\ell^\infty} ds} I_2(t),\qquad \text{ where }\\
&I_2(t)= \left( \int_0^t  T
      \|E_2( \htheta + \tgamma,
  \hrho)\|_{\ell^\infty}^2 ds \right).
\end{split}
\end{equation}
We can then bound the constants by
\[  C_1 = e^{ \frac{1}{2T} \int_0^t ( 8 \| \tgamma \|_{\ell^\infty} + \| 1-4 \tgamma \|_{\ell^\infty} + 2 ) ds}   \leq e^{ \frac32 + \frac32 \delta  + 6 \delta^2}  \leq 8,\] 
\[ C_2 =  e^{1} < 3 ,\]
and 
\[   2  + 3\delta + 24 \delta^2 \leq 3 \]
for $\delta$ sufficiently small.  
\end{lem}

\subsection{Main result}\label{subsect:main}
We proceed with a bootstrap argument to prove uniform bounds for
$(\htheta, \hrho)$.  We will control the error terms with respect to
the parameter $\epsilon$ for a grid of size $N$ up to a time $T =
\epsilon^{-1}$.

\begin{thm}\label{thm:main1alt}
  There exists an $1/16 > \delta > 0$, such that for any $0 < \epsilon$ sufficiently small, $N > 0$, given initial data \eqref{e:pertid} depending upon
  $\epsilon$ for equations \eqref{e:hydroalt1}-\eqref{e:hydroalt2},
  the solution $(\theta,\rho)$ to \eqref{e:hydroalt1},
  \eqref{e:hydroalt2} satisfies
$$\| \theta - \ttheta \|_{L^\infty \ell^\infty } \leq \frac{\delta}{8} , \ \  \|\rho- \trho \|_{L^\infty \ell^\infty} \leq   \frac{\delta \epsilon}{32} $$
for all $0 \leq t \leq T = \delta \epsilon^{-1} $.  Here, $(\ttheta
(t), \trho (t))$ satisfy
\eqref{e:modburgers} with initial data \eqref{e:pertid}. \end{thm}

\begin{proof}
  Let us take a $0< \delta < 1/16$ to be chosen later.  We define $B_{r}
  (0)$ denote the closed ball of radius $r$ centered at $0$ in
  $L^\infty([0,T]; \ell^\infty)$, and hence take the bootstrap
  assumption to be that
$$(\htheta,\hrho) \in B_{ \frac{\delta}{8}} (0)\times B_{ \frac{\delta \epsilon}{32}} (0)$$ 
for all time for a given $\delta$ to be chosen later. In other words,
we will assume the following bounds
\begin{equation}
  \label{eqn:gronwallbootstrapalt}
  \| ( \htheta )\|_{L^\infty_t \ell^\infty} \leq \frac{\delta}{8} , \ \ \| ( \hrho )\|_{L^\infty_t \ell^\infty} \leq \frac{\delta  \epsilon}{32}.
\end{equation}

Let us note here that similar to the computation in
\eqref{eqn:gammabd}, we have
\begin{equation}
  \label{eqn:expgronbdalt}
  e^{24 \int_0^T \| \trho_j \|_{\ell^\infty} (s) ds}  \leq
 e^{  24 T \epsilon \left( \frac{1}{8}  + \frac{\delta}{2} \right) } \leq e^{4 \delta}
\end{equation}
for $T \leq \delta \epsilon^{-1}$.  Due to the bootstrap assumption we
have for $t \leq \delta \epsilon^{-1}$ the bound
$$\max_{j }|\htheta_j|+|\tgamma_j| \leq  \frac{2 \delta}{8} + \delta^2 \leq \frac{1}{2^4}$$
with $\delta$ chosen sufficiently small. 
Hence, using $\epsilon \ll \delta < 1$, we can write the explicit bounds
\begin{align*}
& |E_{1,j}| \leq 16 \| \hrho \|_{L^\infty} \| \htheta \|_{L^\infty} + | \tilde{E}_{1,j}| < \frac{\epsilon \delta^2}{16} + 8 \| \hrho \|_{L^\infty}  \frac{4}{3} \| \htheta + \tgamma \|^3_{L^\infty}  + 8 \| \trho \|_{L^\infty}  \frac{4}{3} \| \htheta + \tgamma \|^3_{L^\infty} \\
& \hspace{1.5cm}  < \frac{\epsilon \delta^2}{16} + \frac{\epsilon \delta}{3} \left(  \frac{  \delta}{ 2^2} + \delta^2 \right)^3 + 8 \left( \frac{\epsilon}{8} + \frac{\epsilon \delta}{2} \right)   \frac{4}{3}  \left(  \frac{  \delta}{ 2^2} + \delta^2 \right)^3 \\
& \hspace{2cm} <  \frac{\epsilon \delta^2}{16}  +  4 \epsilon \left(  \frac{  \delta}{ 2^2} + \delta^2 \right)^3  < \frac{3\epsilon \delta^2}{2^4}, \\
& |(f)_{1,j} | \leq  16 \| \trho \|_{L^\infty} \| \tgamma \|_{L^\infty} <  16 \left( \frac{\epsilon}{8} + \epsilon \delta \right)  \left( \frac{\delta}{8} + \delta^2 \right) \\
& \hspace{1cm} \leq  \epsilon \left( \frac{ \delta}{4} + 4 \delta^2 + 16 \delta^3 \right)
\end{align*}
and
\begin{align*}
& |E_{2,j}|  \leq 8 \| \hrho \|_{L^\infty}^2 + | \tilde{E}_{2,j}| <  \frac{\delta^2 \epsilon^2}{2^7} +   \frac{ \epsilon^2}{4} \left( \frac{ \delta  }{2^2} + \delta^2 \right)^2  < \frac{3\delta^2 \epsilon^2}{2^7}
\end{align*}
for $\epsilon$ sufficiently small.
Here, we have taken the coefficients large enough to dominate algebraic contributions of each term in the expansions for $f_{1,2}$ and $E_{1,2}$ with the contribution from $\tilde E_{1,2}$ doubled in order to bound all higher
order terms by twice the worst bound on those of lowest order.
Careful control of such error terms will most definitely allow for somewhat
sharper bounds, however these terms are much lower order compared to
boundary effects, so we do not work carefully to optimize them.

Thus, we deduce that there exists $C > 0$ (can be fixed uniformly in
$\epsilon$ for $\delta$ small) such that
\begin{align*}
  |\text{r.h.s.\ of }\eqref{eqn:bootstrap1}|  & \leq  8  T^2 \epsilon^2 \left[  \frac{3}{16^2} \delta^2 + 2\left( \frac{ \delta}{4} + 4 \delta^2 + 16 \delta^3 \right)^2 + 2 \frac{9\delta^4}{2^8} \right] \\
  & \hspace{1cm} \leq C \delta^4 < \frac{\delta^2}{64}
  \end{align*}
and
\begin{align*}
  |\text{r.h.s.\ of }\eqref{eqn:bootstrap2}| & \leq 3 T^2 \epsilon^2 [ \frac{9 \delta^4 \epsilon^2}{2^{14}}  ]  \\
  & \hspace{.5cm} \leq \frac{\epsilon^2 \delta^2 }{32^2} 
    (2 \delta^4) <  \frac{
    \delta^2 \epsilon^2}{32^2}
\end{align*}
for $ \delta$ sufficiently small.
%Specifically, we compute crudely that for $\sqrt{\epsilon} < \delta < 1/10$, we have$C < 1.538$ and hence the estimate for \eqref{eqn:bootstrap1} will close.  
Note that $\delta$ is chosen independently of
$\epsilon$ and $N$.  As a result, we can close the bootstrapping
argument in $\htheta$ and $\hrho$ independent of our choice of lattice size $N$ and any initial
step size $\epsilon/8$. 
\end{proof}

\begin{rem}
  \label{rem:rescale2}
  By nature of the construction, we now have that
  \begin{equation*}
    |b_j|^2 (t) = \rho_j (t) \leq \trho_j (t) + \left(\frac{\epsilon \delta}{32} \right).
  \end{equation*}
  Hence, using the envelope estimates from Remark \ref{rem:rescale} for the explicit rarefaction wave in Section \ref{subsect:gronw1} and Theorem \ref{thm:main1alt}, near the end of our evolution the rarefaction wave at the
  left of the grid has size roughly bounded
  \begin{equation*}
   | \rho_0 (T)| < \frac{\epsilon}{8}  - \frac{\epsilon\delta}{10},\end{equation*}
  which is less than the initial amplitude $\epsilon/8$ even when compared to the error.   A symmetric argument using the growth of $\rho_N$ shows that it has increased by a non-trivial amount.
 Hence, we observe that our method moves mass towards the right and given that the symmetric Burgers solution increases at the right endpoint, the Toy Model initially does as well.  Obviously we would like a much stronger proof of mass transfer by rarefaction wave dynamics, which remains an open problem relating to the global structure of the mass transfer in the full Toy Model.  
  
\end{rem}

\section{Additional Observations and Remarks}
\label{sect:rem}

\subsection{Flux Computation for Finite Approximations}
\label{sec:flux}

We have the Hamiltonian system
\begin{equation}
  \left\{ \begin{array}{c}
      -i \partial_t b_j  = -| b_j|^2 b_j + 2 b_{j-1}^2 \bar{b}_j + 2 b_{j+1}^2 \bar{b}_j, \\ 
      i \partial_t \bar{b}_j  = -| \bar{b}_j|^2 \bar{b}_j + 2 \bar{b}_{j-1}^2 b_j + 2 \bar{b}_{j+1}^2 b_j,
    \end{array} \right.
\end{equation}
which is only Hamiltonian with respect to the infinite sum unless we
are certain that our initial data is compactly supported.  However, as
suggested to us by Jonathan Mattingly \cite{matt-pc} based off of
ideas in \cite{MSV-cascade}, let us take initial data supported on the
infinite half lattice, yet restrict the Hamiltonian system to the
first $N$ nodes and simply look at the flux in the energy at this
sufficiently high node, where we now have
\begin{equation}
  H_N = \sum_{j=1}^N \frac14 |b_j|^4 - \Re ( \bar{b}_j^2 b_{j-1}^2),
\end{equation}
which now is not perfectly conserved.  Indeed, we have
% \begin{equation}
%\partial_t H_N = -2 |b_N|^2 \Im ( b_{N-1}^2 \bar{b}_{N+1}^2 - \frac12 b_N^2 b_{N+1}^2).
%\end{equation}
\begin{equation}
  \partial_t H_N = 2 |b_N|^2 \Im ( 2 b_{N+1}^2 \bar{b}_{N-1}^2 -   b_{N+1}^2 \bar{b}_{N}^2).
\end{equation}

Moving from the exact formula to do some asymptotic analysis, if we
assume roughly comparable amplitude (Note: we believe is the case at
say $j = \frac{N}{2}$ up to the time the rarefaction wave and
dispersive shock meet) of the final three nodes gives
\begin{equation}
  \partial_t H_N \approx  A_N^6 [4 \sin (2 \phi_{N+1} -2 \phi_{N-1}) - 2 \sin (2 \phi_{N+1} -2 \phi_{N}) ]
\end{equation}
taking $b_j= A_j e^{i \phi_j}$.  Hence, the Hamiltonian flux is seen
to be positive (inward flow of energy) if
\begin{equation} [ \sin (2 \phi_{N+1} -2 \phi_{N-1}) - \frac12 \sin (2
  \phi_{N+1} -2\phi_{N}) ] > 0.
\end{equation}
% which holds for instance if
%$$0 < 2 \phi_{N+1} - 2\phi_{N-1} < \pi , \ \ -\pi < 2 \phi_{N+1} - 2 \phi_{N}  < 0,$$
%or as would be the case for $\phi_{j+1} = \frac{(j-1) \pi}{4}$ if $N$
% is even.
In order to see outward flow of energy, the Hamiltonian flux is
negative if
\begin{equation} [ \sin (2 \phi_{N+1} - 2\phi_{N-1}) - \frac12 \sin (2
  \phi_{N+1} - 2 \phi_{N}) ] < 0.
\end{equation}
% which holds for instance if
%$$-\pi < 2 \phi_{N+1} - \phi_{N-1} < 0 , \ \ 0 < 2 \phi_{N+1} - \phi_{N}  < \pi.$$
%or as would be the case for $\phi_{j+1} = \frac{(j-1) \pi}{4}$ if $N$
% is odd.
For $\phi_{j} = \frac{(j-1) \pi}{4}$, we observe
\begin{equation} [ \sin (2 \phi_{N+1} - 2\phi_{N-1}) - \frac12 \sin (2
  \phi_{N+1} - 2 \phi_{N}) ] = -3/2 ,
\end{equation}
which would actually result in an outward flow of energy, though as
the Hamiltonian energy is not coercive, we do not gain much from this
computation.

Alternatively, we look at the restricted mass flux,
\begin{equation}
  M_N = \sum_{j=1}^N |b_j|^2,
\end{equation}
which now is not perfectly conserved.  Indeed, we have
\begin{equation}
  \partial_t M_N = -4  \Im ( b_{N+1}^2 \bar{b}_N^2).
\end{equation}
Making a similar asymptotic assumption at the endpoint, we have
\begin{equation}
  \partial_t M_N \approx  -4  |A_N|^4 \sin ( 2 (\phi_{N+1} - \phi_N)).
\end{equation}
Hence, we observe that the mass flux is {\it outgoing} for $\phi_{j} =
\frac{(j-1) \pi}{4}$ since we then have
\begin{equation}
  \partial_t M_N \approx  -4  |A_N|^4 \sin ( \frac{\pi}{2}).
\end{equation}
% Later we will study rarefaction wave solutions in the full toy model
% and re-visit the mass flux computation using specific solutions.

\begin{rem}
  \label{rem:massflux}
  As pointed out by the anonymous referee, we can use the mass flux
  computation from Section \ref{sec:flux} to study the rarefaction
  wave solution from Theorem \ref{thm:main1alt}.  In such a case, we
  observe that for say the node $j=N/2$, which remains roughly fixed
  at $\rho_j (t) \sim \epsilon/8$ and $\theta_j (t) \sim \pi/4$, we
  have the total mass moved across this node of order
$$ \epsilon^{-1} \epsilon^4 = \epsilon^3.$$
Initially, this looks like a rather small mass flux compared to the
size of the solution overall.  However, we note first of all that in
this setting $\epsilon$ need not necessarily be extremely small since
the asymptotic methods are are done mostly on the side of the $\delta$
parameter. In addition, we note that numerically of course, the
mid-point of the rarefaction wave solution remains stable much longer
than the time scales we have controlled here.  Indeed, while the
rarefaction wave moving left and the dispersive shock like solution
moving right definitely change the structure of the backwards Burgers
equation on a time scale of order $\epsilon^{-1}$, we have strong
numerical evidence that away from the fronts of those waves the
solution remains largely unchanged.  Hence, we expect that with
greater global control over the dynamics, the mass flux computation
can be shown to be much stronger than can be applied on the time
scales in Theorem \ref{thm:main1alt}.
\end{rem}

\begin{rem}
  \label{rem:rescale1}

  As commented in \cite{CMOS1}, this analysis still leads to open
  questions about Sobolev norm growth in the full problem
  \eqref{e:dcnls} given the pointwise bounds on the error for $j \sim
  N/2$ on the same time scale, we observe that the flux computation in
  Section \ref{sec:flux} will continue moving mass towards high $j$ on
  this time scale.  In addition, computational checks of the constants
  suggest that the bootstrapping arguments in Theorem
  \ref{thm:main1alt} appear to go through for $\delta$ chosen even as
  large as $1/2$, meaning while we need our time interval to be
  $o(1)$, there should be parts of the argument that extend to time
  $1$.  Doing so in a rigorous fashion likely requires more analytic
  control on the global structure of the rarefaction wave-like
  solution both in discrete symmetric Burgers and in terms of the
  behavior in the Toy Model near the right boundary.
\end{rem}

\subsection{An observation about $\| \cdot \|_{\ell^2}$ growth of
  $\hrho$, $\htheta$}\label{subsect:obs}

We present here an illustrative computation, which unfortunately at
the moment we cannot apply in a perturbation theoretic argument as we
would require stronger control of the behavior of solutions to
\eqref{e:toy_model} at the endpoints of our finite region.  Let us
recall that
\begin{align*}
  \dot{\hrho}_j = 4\hrho_{j-1} \trho_j - 4 \hrho_j (\trho_{j+1} -
  \trho_{j-1} ) - 4\hrho_{j+1}\trho_j+ F_j,
\end{align*}
with \[F_j=2 \hrho_{j-1} \trho_{j} \tgamma_j^2 - 2\hrho_{j+1}\trho_{j}
\tgamma_{j+1}^2 + (f)_{2,j} (t) + \mathcal{O} (|\hrho_j + \trho_j|^2
|\htheta_j + \tgamma_j|^2).\] Combining terms from nearest neighbors
we have
\begin{align*}
  &\tfrac12 \partial_t {( \hat{\rho}_{j-1}^2)} + \tfrac12 \partial_t
  {(
    \hat{\rho}_{j}^2)} + \tfrac12 \partial_t { (\hat{\rho}_{j+1}^2 )}\\
  & \hspace{.5cm}  = 4\hrho_{j-2}\hrho_{j-1}\trho_{j-1}-4\hrho_{j-1}^2(\trho_j-\trho_{j-2})-4\hrho_{j}\hrho_{j-1}\trho_{j-1}+F_{j-1}\hrho_{j-1}\\
  &\hspace{1.0cm} +4\hrho_{j-1}\hrho_{j}\trho_{j}-4\hrho_{j}^2 (\trho_{j+1}-\trho_{j-1}) -4\hrho_{j+1}\hrho_{j}\trho_{j}+F_{j}\hrho_{j}\\
  & \hspace{1.5cm} +4\hrho_{j}\hrho_{j+1}\trho_{j+1}-4\hrho_{j+1}^2
  (\trho_{j+2}-\trho_{j})
  -4\hrho_{j+2}\hrho_{j+1}\trho_{j+1}+F_{j+1}\hrho_{j+1}.
\end{align*}
By combining nearby terms, we observe that
\begin{align*}
  &-4\hrho_{j-1}^2(\trho_j-\trho_{j-2})-4\hrho_{j}\hrho_{j-1}\trho_{j-1}+4\hrho_{j-1}\hrho_{j}\trho_{j}-4\hrho_{j}^2
  (\trho_{j+1}-\trho_{j-1}) \\
  & \hspace{1cm} -4\hrho_{j+1}\hrho_{j}\trho_{j} +4\hrho_{j}\hrho_{j+1}\trho_{j+1}-4\hrho_{j+1}^2 (\trho_{j+2}-\trho_{j})\\
  & \hspace{.5cm} = -4\hrho_{j-1}^2(\trho_j-\trho_{j-2}) +4
  \hat{\rho}_j \hat{\rho}_{j-1} ( \tilde{\rho}_{j} -
  \tilde{\rho}_{j-1})
  - 4 \hat{\rho}_j^2( \tilde{\rho}_{j+1} - \tilde{\rho}_{j-1}) \\
  & \hspace{1.75cm}+ 4\hat{\rho}_j \hat{\rho}_{j+1}(
  \tilde{\rho}_{j+1} - \tilde{\rho}_{j})
  -4\hrho_{j+1}^2 (\trho_{j+2}-\trho_{j})\\
  & \hspace{1cm} \leq - 2 \hat{\rho}_{j-1}^2 ( \tilde{\rho}_{j} -
  \tilde{\rho}_{j-2})
  -2 \hat{\rho}_j^2 ( \tilde{\rho}_{j+1} - \tilde{\rho}_{j-1}) - 2 \hat{\rho}_{j+1}^2 ( \tilde{\rho}_{j+2} - \tilde{\rho}_{j}) \\
  & \hspace{2cm} \leq 0
\end{align*}
by Cauchy-Schwarz. Summation yields
\[
\tfrac12 \partial_t \sum_{j=1}^{N} \hat{\rho}_{j}^2 \leq 2 \hrho_N^2 (
\trho_N + \trho_{N-1} ) +\sum_{j=1}^{N} F_{j}\hrho_{j}
\]
since the $\hrho$ terms are compactly supported on the interval $j =
1, \dots N$.  Hence, we observe that on the time scale of evolution,
the leading order terms here involve only the right endpoints and are
by construction positive, which fits with the conservation of mass in
\eqref{e:toy_model}.

\subsection{Other Discrete Conservation Laws}\label{subsect:other}
In \cite{HerrmannRademacher}, the authors study Fermi-Pasta-Ulam
Systems of the form
\begin{equation*}
  \dot{r}_j = v_{j+1} - v_j, \\
  \dot{v}_j = \phi' (r_j) - \phi' ( r_{j-1}),
\end{equation*}
where $\phi (r) = e^{1-r} - (1-r)$.  Continuous limits $(r_j, v_j) (t)
= (r,v) (\epsilon t, \epsilon j)$ of such models also satisfy the
Burgers equation
\begin{equation*}
  \partial_t r - \partial_x v = 0, \quad \partial_t v - \partial_x \phi' (r) = 0.
\end{equation*}
Rarefaction waves and dispersive shocks are observed and studied in
depth numerically and some analysis is done on conservative shock
formation.  In future work, we hope similar analysis can be done to
that for \eqref{e:toy_model}.

\section{Numerical Study of Rarefaction and Relationship to the Toda
  Lattice}
  \label{s:numerics}

Following a suggestion of the referee, the authors explored references
\cite{GoodmanLax,GHSZ-Toda,KvM} relating symmetric Burgers 
\eqref{e:burgers} and the Toda Lattice, a known integrable system,
see also the survey article \cite{teschl2001almost}.  The key idea is
that $\sqrt{\rho_j} (t/4)= a_j (t)$, the system becomes the Kac-van
Moerbeke system (KvM)
\begin{equation}
  \label{KvM}
  \dot a_j = - a_j (a_{j+1}^2 - a_{j-1}^2)
\end{equation}
and this has a direct connection via its even and odd modes to the
Toda Lattice by interpreting the $a_j$'s as entries in a Jacobi
Matrix.  In \cite{GHSZ-Toda}, they construct the algebro-geometric
inverse of the Toda Solitons to the KvM system, which translates in
discrete Burgers to solutions on all of $\Z$ of the form
\begin{equation}
  \label{Toda_1}
 \rho_j (z, t) = \gamma \left\{  \begin{array}{l} 
      \frac{  1 + c(t)^2 z^{2m + 2} ( 1 - z^2)^{-1} }{  2 z ( 1 + c (t)^2 z^{2m} (1-z^2)^{-1} }, \ \ j = 2m,  \\
      \frac{  z ( 1 + c(t)^2 z^{2m} (1 - z^2)^{-1} ) }{  2 ( 1 + c(t)^2 z^{2m + 2} ( 1- z^2)^{-1} }, \ \ j = 2m + 1, \end{array} \right.
\end{equation}
where
\[ c(t) = c_0 e^{ \frac{(z - z^{-1}) t }{2}}, \] where the constant
$\gamma$ determines the total amplitude and partially the time of
evolution compared to the (KvM) system.  The rarefaction wave can be
seen as the setting where $z = 1$.  Otherwise, this solution gives
interesting splitting of the even and odd modes that we will show
remains somewhat stable in the toy model.

 \begin{figure}
   \subfigure{\includegraphics[width=4cm,type=pdf,ext=.pdf,read=.pdf]{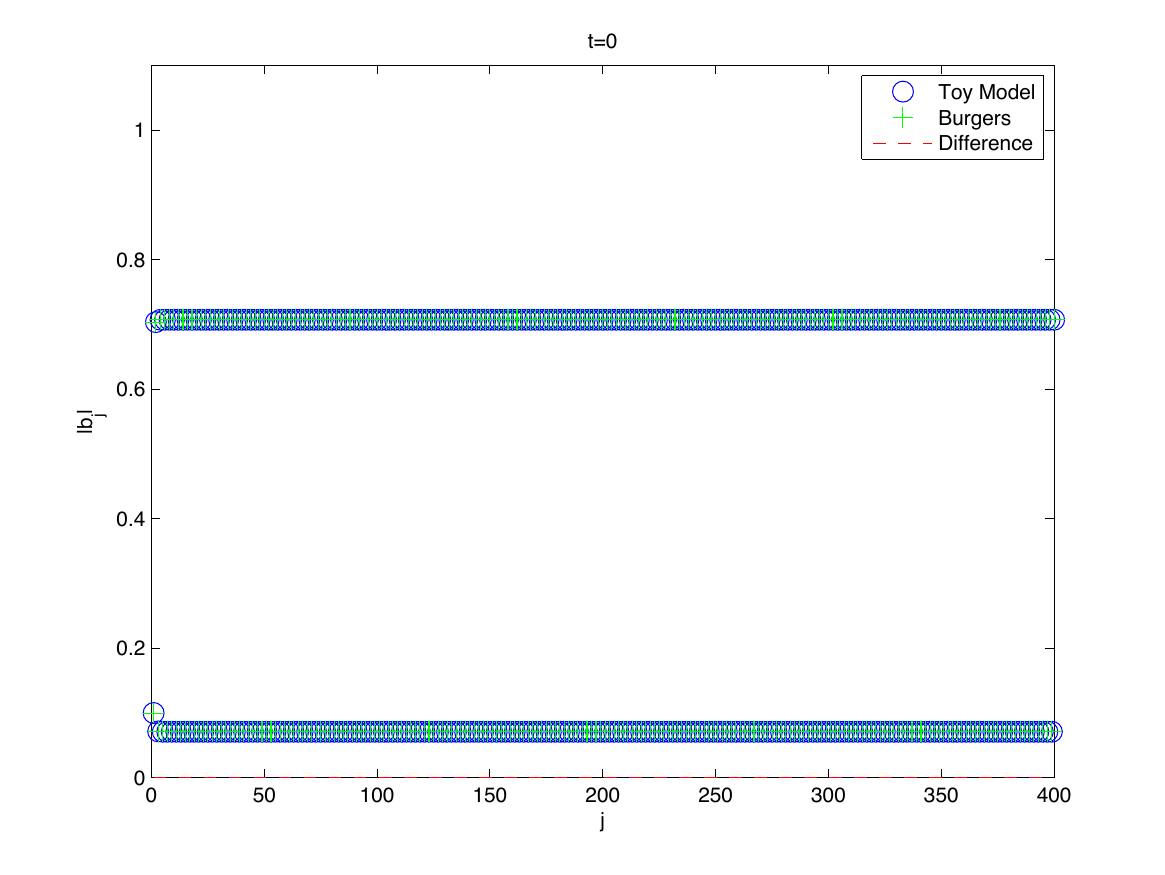}}
   \subfigure{\includegraphics[width=4cm,type=pdf,ext=.pdf,read=.pdf]{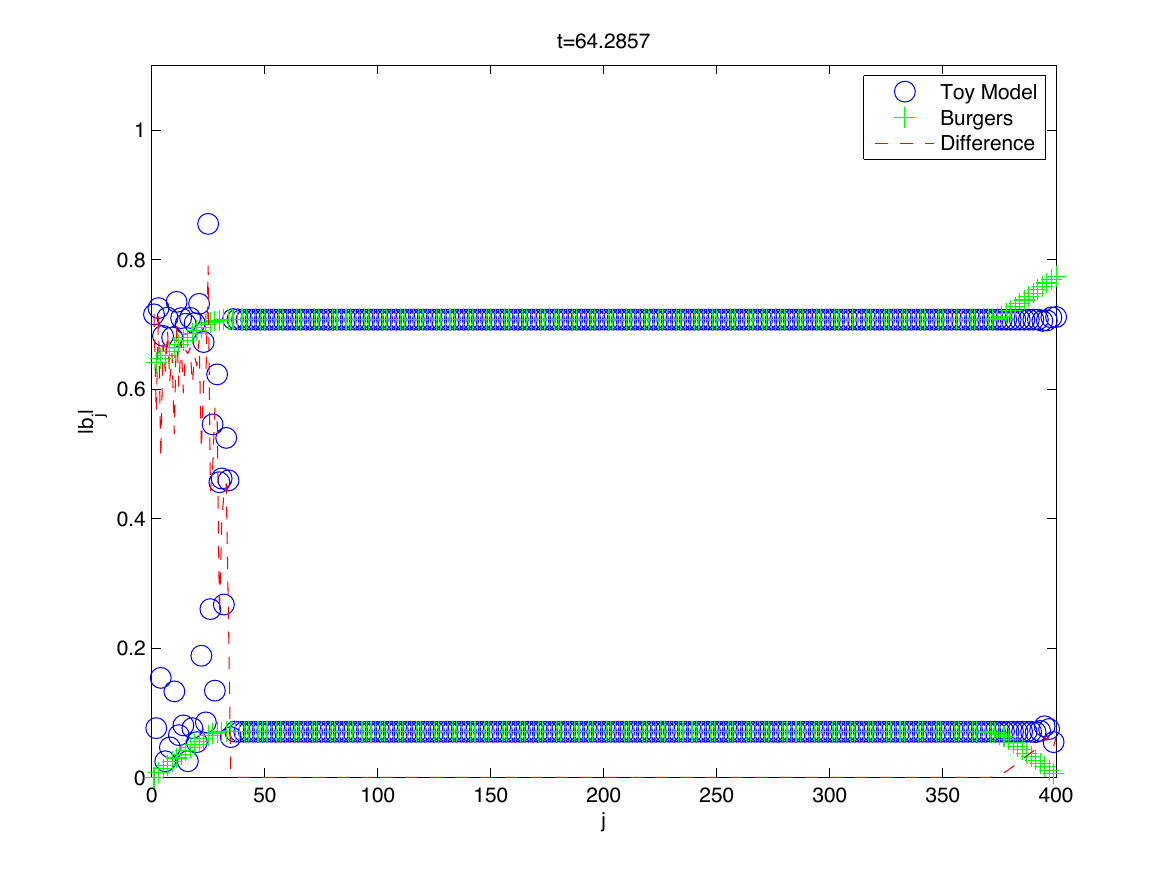}}
   \subfigure{\includegraphics[width=4cm,type=pdf,ext=.pdf,read=.pdf]{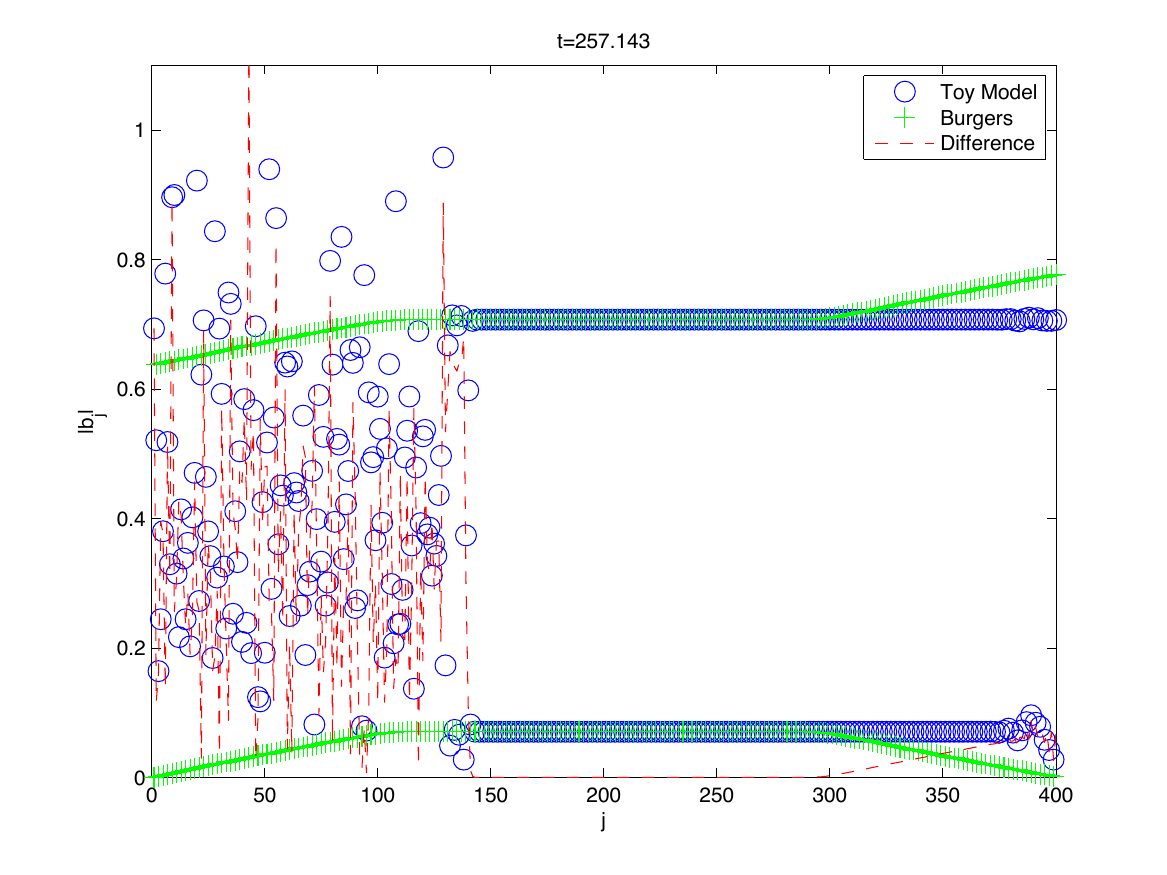}}
   \subfigure{\includegraphics[width=4cm,type=pdf,ext=.pdf,read=.pdf]{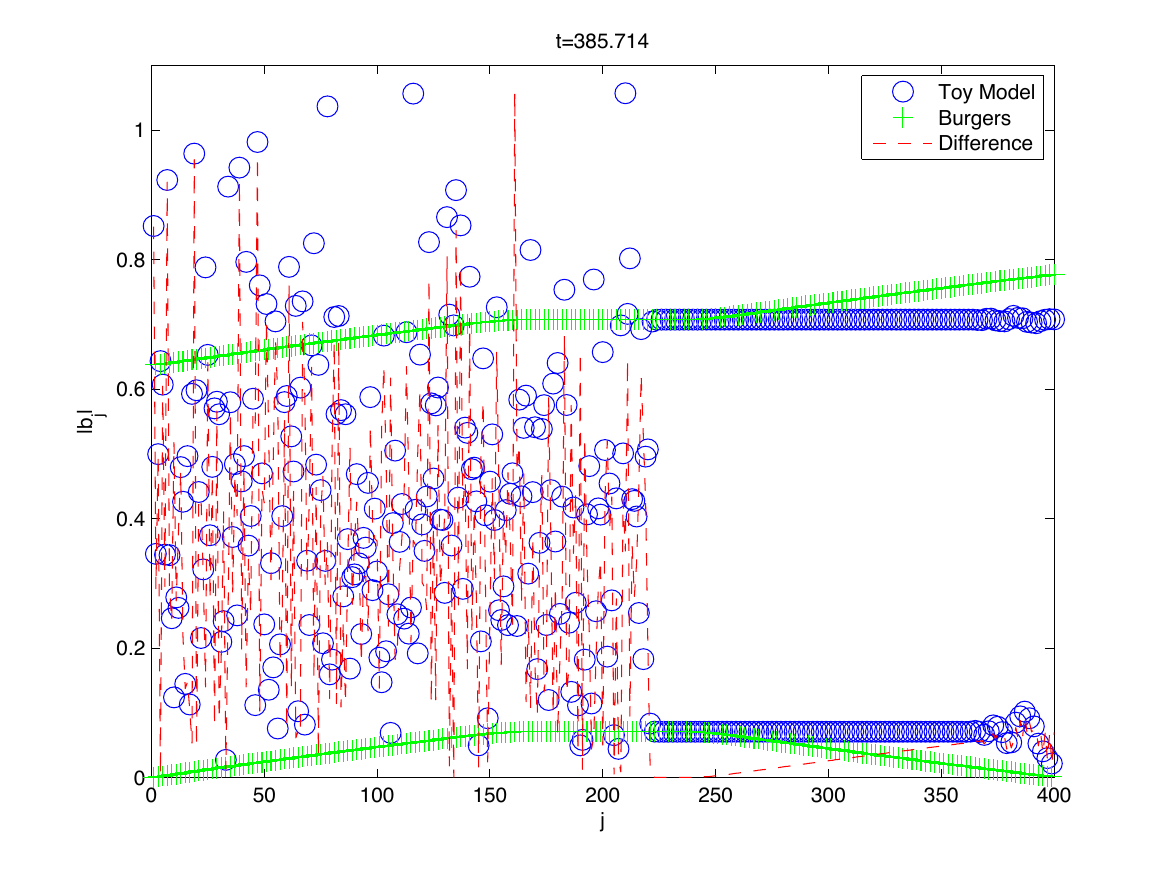}}
   \subfigure{\includegraphics[width=4cm,type=pdf,ext=.pdf,read=.pdf]{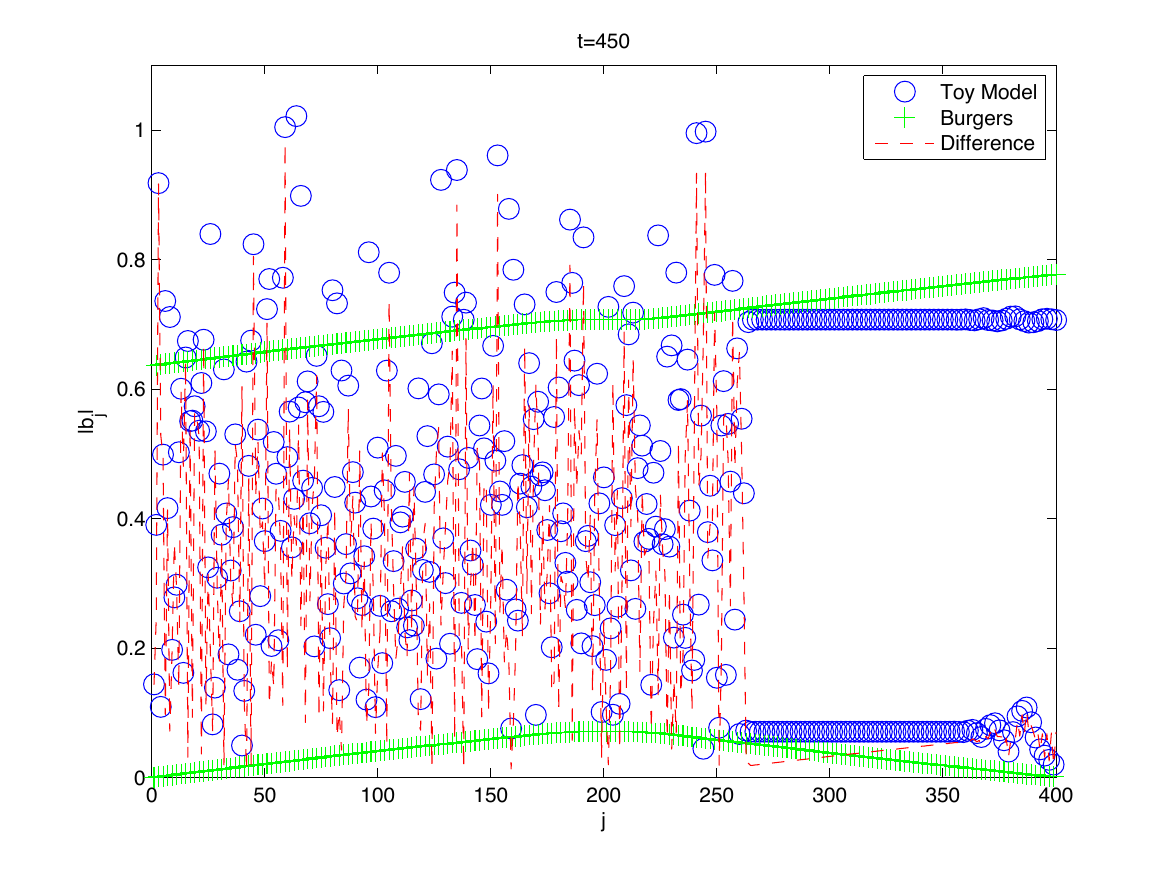}}
   \caption{A comparison of the evolution of solutions stemming from
     an initial amplitude profile given by a truncated Toda soliton
     solution with parameters $\gamma = .1$, $z = .1$ and $c_0 =
     1.0$.}
   \label{todasol}
 \end{figure}

First, we numerically have studied the evolution of out of phase
initial data with amplitude given by the mapped $1$-soliton solution
wave truncated to a domain of $400$ nodes with $\gamma = .1$, $z = .1$
and $c_0 = 1.0$ in Figure \ref{todasol}.  As one can observe, there
are some interesting features in the Toy Model as well as in the
Burgers solution, but backward propagation from the truncation on the
right and drag in the phase cause the solution to differ significantly
on a similar time scale to the rarefaction wave.

 \begin{figure}
   \subfigure{\includegraphics[width=6cm,type=pdf,ext=.pdf,read=.pdf]{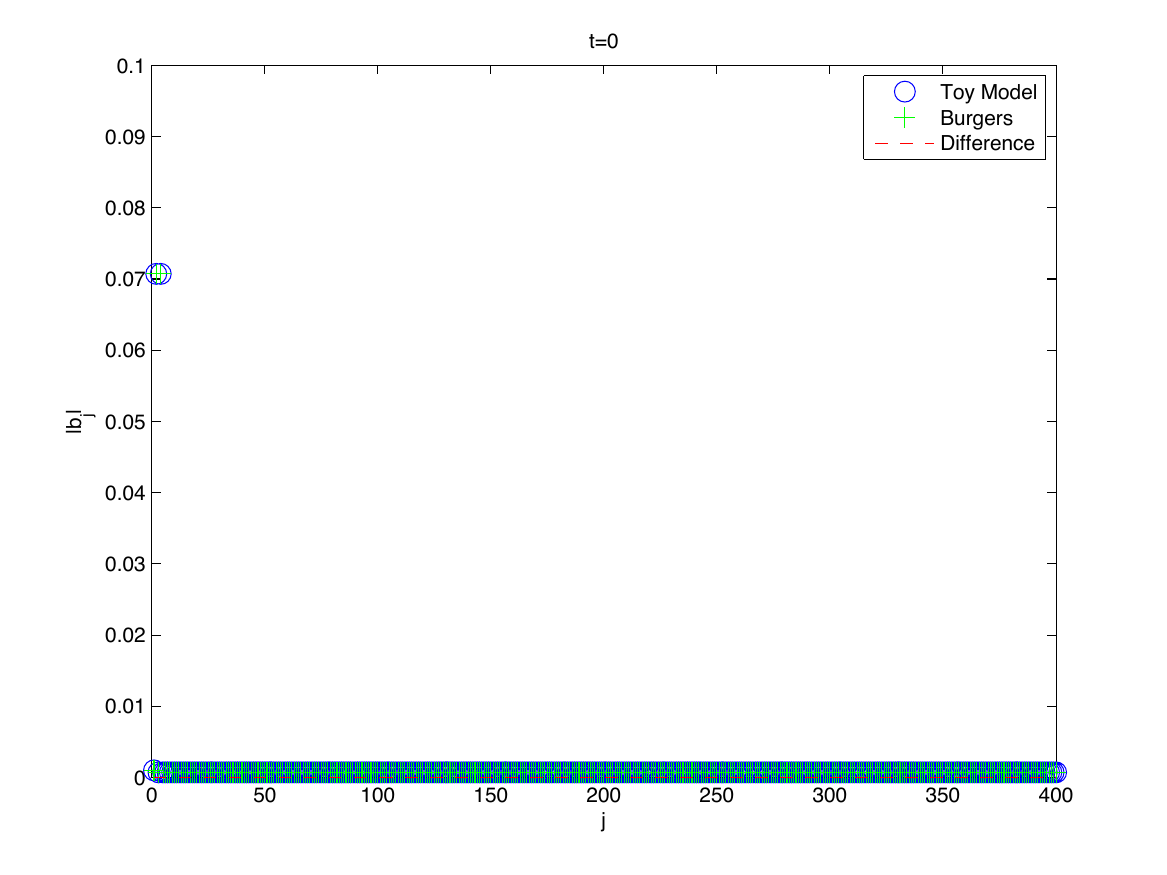}}
   \subfigure{\includegraphics[width=6cm,type=pdf,ext=.pdf,read=.pdf]{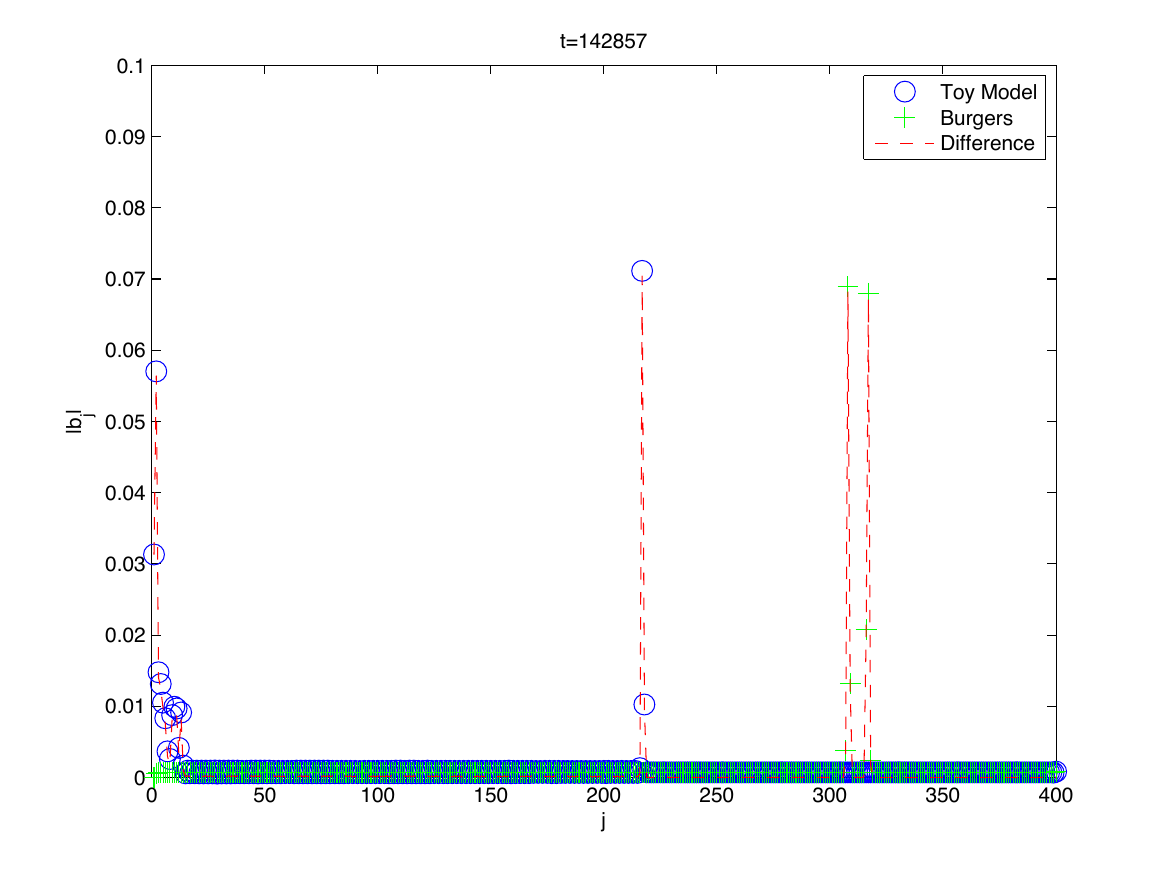}}
   \subfigure{\includegraphics[width=6cm,type=pdf,ext=.pdf,read=.pdf]{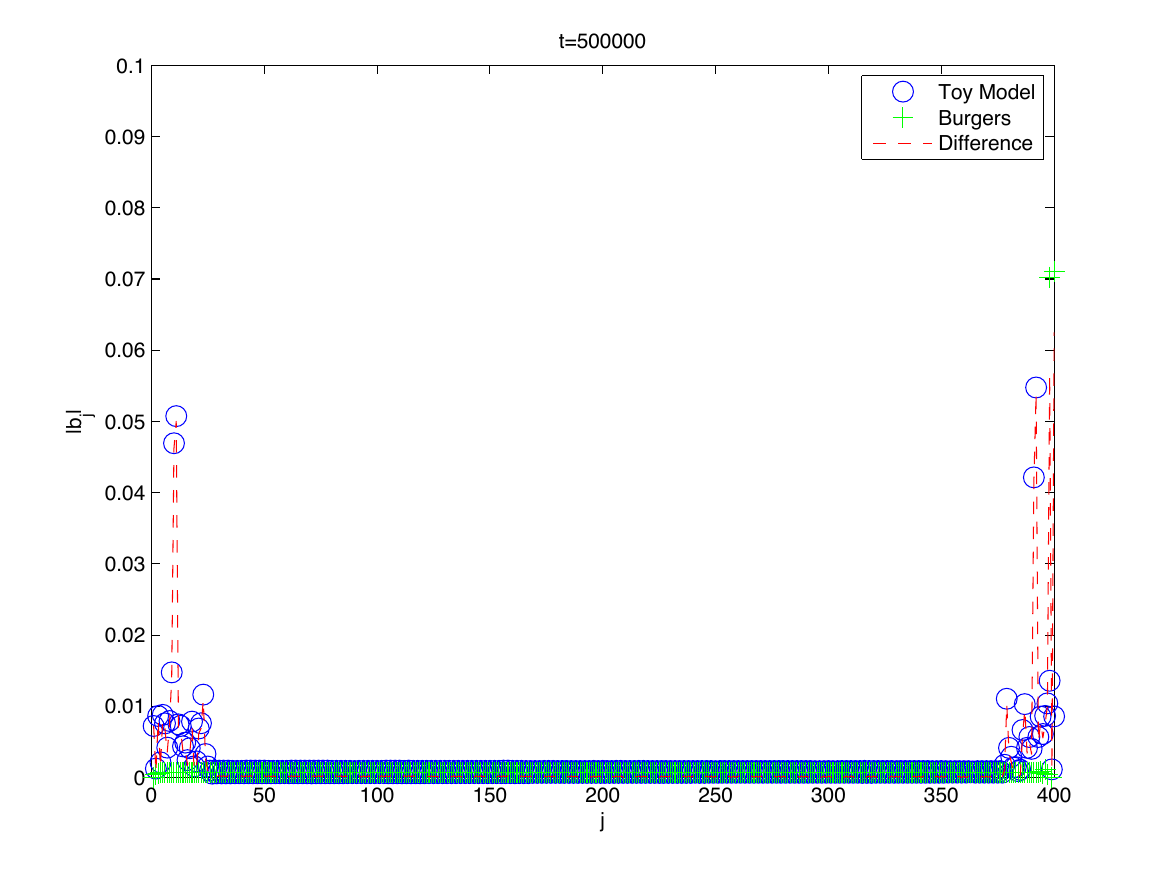}}
   \caption{A comparison of the evolution of solutions stemming from
     an initial amplitude profile given by applying the Toda soliton
     solution on an even number of nodes at the left, then continuing
     on the right as a small non-zero constant with parameters $\gamma
     = .0001$, $z = .01$ and $c_0 = 1.0$.}
   \label{f:todatravel}
 \end{figure}

However,
based upon the insight that the Toda lattice soliton can generate mass
exchange, in Figure \ref{f:todatravel}, we then took out of phase
initial data corresponding to \eqref{Toda_1} on the first $M = 4$
nodes, then continued by $\rho_j = z/2$ for $j >M$.  Interestingly
enough, this initial data generates right traveling nodes at fast time
scales in the Burgers model and as seen in Figure \ref{f:todatravel}
seems to lead to a slow, but steady oscilating traveling wave in the
Toy Model.  The peaked wave observed there persists on longer domains
and with decaying tails to the right (provided the decay is slow).
Also, a peaked wave of this form can be observed choosing $M$ as small
as $6$, but there does not appear to be enough energy if $M =2$ or
$4$.  The dynamics are generally on a much slower scale than that of
the comparable symmetric Burgers equation, but the traveling wave
observed is rather intriguing as a potential for mass transfer.

Lastly, in Figure \ref{rarefaction1}, we studied the evolution of a
simple out of phase rarefaction wave on $N = 2000$ nodes with $\rho_j
= 1$ on the first $M = 100$ notes, then a decaying tail of the form
$\rho_j = 1/(1 + .1 (j-100)^{\frac12}$ for $j > 100$.  The general
flow of the mass can be seen to be to the right, though of course in
such a setting we have no envelope equation with which to try to
approximate the dynamics.
 \begin{figure}
   \subfigure{\includegraphics[width=6cm,type=pdf,ext=.pdf,read=.pdf]{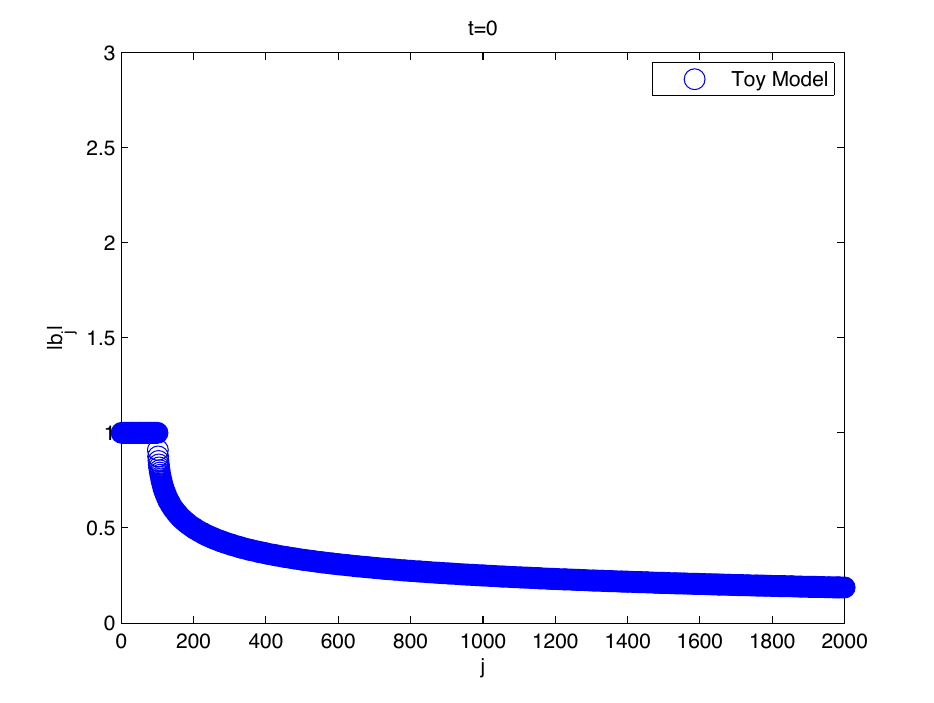}}
   \subfigure{\includegraphics[width=6cm,type=pdf,ext=.pdf,read=.pdf]{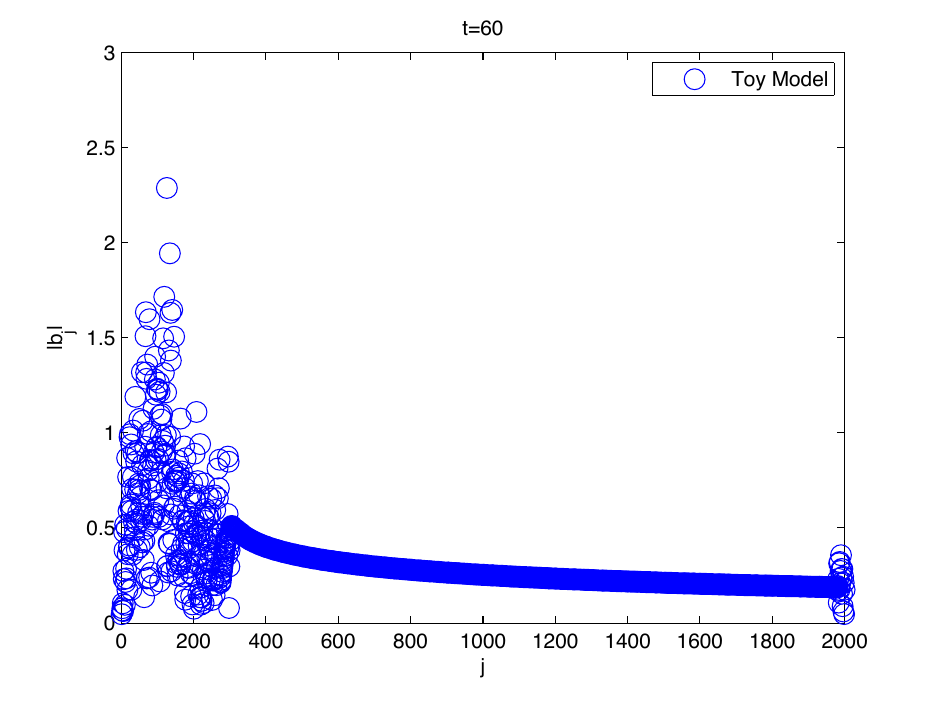}}
   \subfigure{\includegraphics[width=6cm,type=pdf,ext=.pdf,read=.pdf]{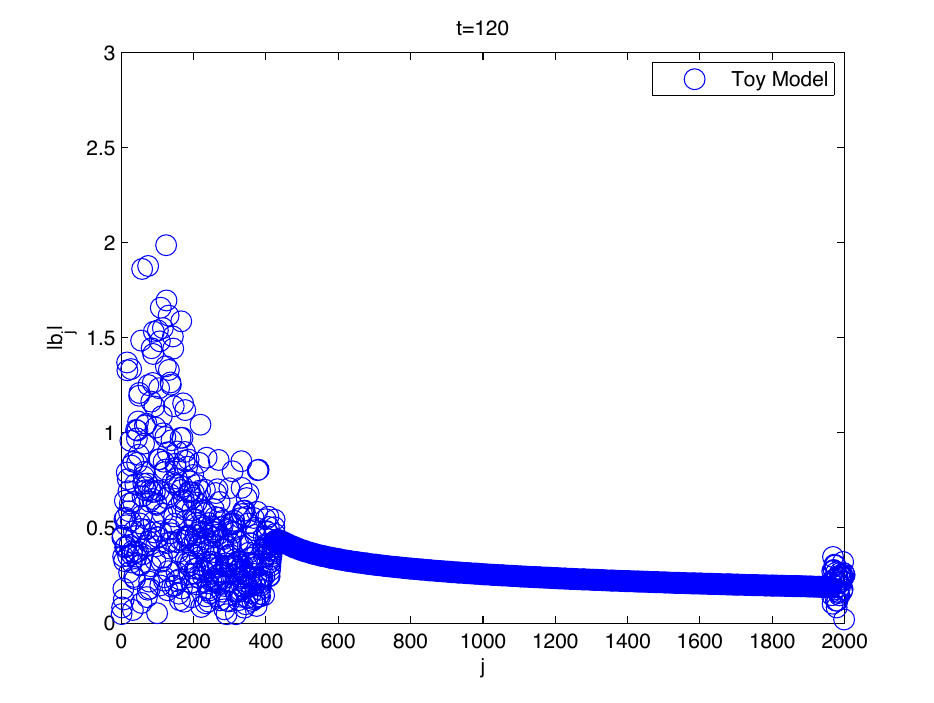}}
   \subfigure{\includegraphics[width=6cm,type=pdf,ext=.pdf,read=.pdf]{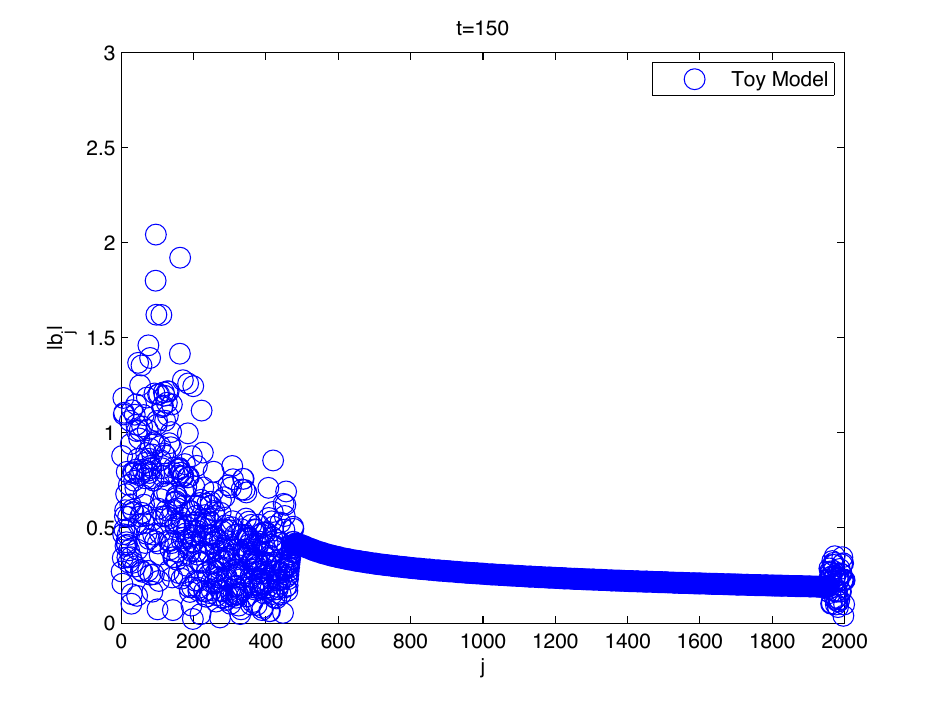}}
   \caption{A comparison of the evolution of solutions for a
     rarefaction wave profile with a decaying tail of form
     $1/(1+.1*j^{\frac12})$ to minimize the backward propagating shock
     from the right boundary.}
   \label{rarefaction1}
 \end{figure}

\subsection*{Acknowledgements}
The first author was supported by the German Research Foundation, CRC
701.  The second author was supported by a combination of an IBM
Junior Faculty Development Award through the University of North
Carolina, NSF Grant 1312874, plus Guest Lectureships through
Universit\"at Bielefeld Summer 2012 and Karlsruhe Institute of
Technology in Summer 2013.
  
Thanks especially to Gideon Simpson for many helpful conversations
about this topic and for allowing us to modify his code to create the
pictures seen here.  The authors also wish to thank Jianfeng Lu, Hiro
Oh and Jonathan Mattingly for interesting discussions, but
particularly Mattingly for suggesting the flux computation displayed
in Section \ref{sec:flux}.  In addition, the authors thank the
anonymous referees for pointing out many places the exposition could
be improved or clarified and in particular for such a careful reading
of the article.  Also, for pointing out the connection to the Toda
Lattice, which proved invaluable.
  
\bibliographystyle{abbrv}

\bibliography{ToyModel2}

\end{document}

%% file: gsmacros.tex
\newcommand{\abs}[1]{\lvert#1\rvert}

\newcommand{\set}[1]{\left\{#1\right\}}

\newcommand{\dt}{\partial_{t}}

\newcommand{\Z}{\mathbb{Z}}

\newcommand{\trho}{\tilde{\rho}}
\newcommand{\ttheta}{\tilde{\theta}}
\newcommand{\hrho}{\hat{\rho}}
\newcommand{\htheta}{\hat{\theta}}
\newcommand{\hsigma}{\hat{\sigma}}
\newcommand{\tsigma}{\tilde{\sigma}}
\newcommand{\tgamma}{\tilde{\gamma}}

\newtheorem{thm}{Theorem}[section]
\newtheorem{prop}{Proposition}[section]

\newtheorem{lem}[prop]{Lemma}

\newtheorem{rem}[prop]{Remark}